\documentclass[12pt]{amsart}

\usepackage[matrix,arrow,curve,frame]{xy}
\usepackage{amsmath,amsthm,amssymb,enumerate}
\usepackage{latexsym}
\usepackage{amscd}
\usepackage[colorlinks=false]{hyperref}
\usepackage{euscript}

\newtheorem{thm}[subsection]{Theorem}

\newtheorem{lemma}[subsection]{Lemma}
\newtheorem{conj}[subsection]{Conjecture}
\newtheorem{remark}[subsection]{Remark}

\theoremstyle{definition}
\newtheorem{example}[subsection]{Example}

\numberwithin{equation}{section}

\def\ra{\rightarrow}
\def\bra{\langle}
\def\ket{\rangle}

\def\cA{{\mathcal A}}
\def\cB{{\mathcal B}}
\def\cC{{\mathcal C}}
\def\cD{{\mathcal D}}
\def\cE{{\mathcal E}}

\def\cO{{\mathcal O}}

\def\cR{{\mathcal R}}
\def\cS{{\mathcal S}}

\def\cV{{\mathcal V}}
\def\cW{{\mathcal W}}

\def\gg{{\mathfrak g}}

\def\gl{{\mathfrak l}}

\def\gp{{\mathfrak p}}

\def\gs{{\mathfrak s}}

\newfont{\german}{eufm10}

\begin{document}
\pagestyle{plain}

\title
{A commutant realization of $\cW_n^{(2)}$ at critical level}

\author{Thomas Creutzig, Peng Gao, and Andrew R. Linshaw}
\address{Fachbereich Mathematik, Technische Universit\"at Darmstadt.}
\email{tcreutzig@mathematik.tu-darmstadt.de}

\address{Simons Center for Geometry and Physics, Stony Brook University.}
\email{pgao@scgp.stonybrook.edu}

\address{Department of Mathematics, Brandeis University.}
\email{linshaw@brandeis.edu}
%\thanks{}

%\date{\today}
{\abstract

\noindent
For $n\geq 2$, there is a free field realization of the affine vertex superalgebra $\cA$ associated to $\gp \gs\gl(n|n)$ at critical level inside the $bc\beta\gamma$ system $\cW$ of rank $n^2$. We show that the commutant $ \cC = \text{Com}(\cA,\cW)$ is purely bosonic and is freely generated by $n+1$ fields. We identify the Zhu algebra of $\cC$ with the ring of invariant differential operators on the space of $n\times n$ matrices under $SL_n \times SL_n$, and we classify the irreducible, admissible $\cC$-modules with finite dimensional graded pieces. For $n\leq 4$, $\cC$ is isomorphic to the $\cW_n^{(2)}$-algebra at critical level, and we conjecture that this holds for all $n$.}

\maketitle
%\tableofcontents

\section{Introduction}

Let $\cV$ be a vertex algebra, and let $\cA$ be a subalgebra of $\cV$. The {\it commutant} of $\cA$ in $\cV$, denoted by $\text{Com}(\cA,\cV)$, is the subalgebra consisting of all elements $v\in\cV$ such that $ [a(z),v(w)] = 0$ for all $a\in \cA$. This construction was introduced by Frenkel and Zhu \cite{FZ}, generalizing earlier constructions in representation theory \cite{KP} and physics \cite{GKO}, and is important in the construction of coset conformal field theories. If $\cA$ acts semisimply on $\cV$, $\text{Com}(\cA,\cV)$ can often be studied by decomposing $\cV$ as an $\cA$-module. Otherwise, there are few existing techniques for studying commutant vertex algebras, and there are very few examples where an exhaustive description can be given in terms of generators, operator product expansions (OPEs), and normally ordered polynomial relations among the generators.

An equivalent definition of $\text{Com}(\cA,\cV)$ is the set of elements $v\in \cV$ such that $a(z)\circ_n v(z) = 0$ for all $a\in\cA$ and $n\geq 0$. We may regard $\text{Com}(\cA,\cV)$ as the algebra of  invariants in $\cV$ under the action of $\cA$. If $\cA$ is a homomorphic image of an affine vertex algebra associated to some Lie (super)algebra $\gg$, $\text{Com}(\cA,\cV)$ is just the invariant space $\cV^{\gg[t]}$, and in this case one can apply techniques from invariant theory and commutative algebra. This approach was first used in \cite{LL} in a special case, and was developed more fully in \cite{LSSII}.

In \cite{FS}, Feigin and Semikhatov constructed a remarkable family of vertex algebras $\cW_n^{(2)}$ at level $k\in \mathbb{C}$, where $n\geq 2$ is a positive integer. For $n=2$, $\cW^{(2)}_2$ is isomorphic to the affine vertex algebra $V_k(\gs\gl_2)$, and for $n=3$, $\cW^{(2)}_3$ coincides with the Bershadsky-Polyakov vertex algebra \cite{Ber}\cite{P}. The superscript $2$ indicates that $\cW^{(2)}_n$ is a generalization of $V_k(\gs\gl_2)$, and it has certain features in common. For example, $\cW^{(2)}_n$ is generated by two bosonic fields of conformal weight $n/2$, and is expected to be freely generated by these two fields together with $n-1$ additional bosonic fields of weights $1,2,\dots, n-1$. The authors provide two different constructions of $\cW^{(2)}_n$. First, it is defined as an algebra in the kernel of screenings associated to the quantum super group $\mathcal U_q(\gs\gl(n|1))$. Second, it is realized as a subalgebra of $\text{Com}(V_{k'}(\gs\gl_n)\otimes V(\gg\gl_1),V_{k'}(\gs\gl(n|1))\otimes V_L)$, where $V_L$ is a rank one lattice vertex algebra. Here, the levels are related by $(k'+n-1)(k+n-1)=1$. Special cases of $\cW^{(2)}_n$ also appear as subalgebras of extensions of the affine vertex superalgebra $V_k(\gg\gl(1|1))$ \cite{CRi1,CRi2}. 

Another construction of vertex algebras is the Drinfeld-Sokolov reduction \cite{BT,FF}. These are so-called $\cW$-algebras associated to the affine vertex algebra $V_k(\gg)$ of some Lie algebra $\gg$ and an embedding of $\gs\gl_2$ in $\gg$. For example, the Bershadsky-Polyakov vertex algebra $\cW^{(2)}_3$ can be realized as the $\cW$-algebra associated to $V_{k}(\gs\gl_3)$ for the nonprincipal embedding of $\gs\gl_2$ in $\gs\gl_3$. It is expected that the $\cW_n^{(2)}$-algebra of level $k$ is a similar quantum reduction of $V_{k}(\gs\gl_n)$, although this remains an open question for $n\geq 4$.

In this paper, we begin with a free field realization of the affine vertex superalgebra associated to $\gp \gg\gl(n|n)$ at critical level. It is realized as a subalgebra of the $bc\beta\gamma$ system $\cW$ of rank $n^2$. This free field realization is very similar to the realizations of the affine vertex superalgebra $V_k(\gg\gl(n|n))$ \cite{CR}.
The derived subalgebra of $\gp \gg\gl(n|n)$ is $\gp \gs\gl(n|n)$. Letting $\cA$ be the image of the affine vertex superalgebra of $\gp \gs\gl(n|n)$ at critical level in $\cW$, we are interested in $\cC = \text{Com}(\cA,\cW)$. Our first step is to show that $\cC$ is purely bosonic and can be identified with the commutant of $V_{-n}(\gs\gl_n)\otimes V_{-n}(\gs\gl_n)$ inside the $\beta\gamma$ system of rank $n^2$. This problem fits precisely into the framework developed in \cite{LSSII}, and is closely related to rings of invariant polynomial functions on arc spaces. We find a minimal strong finite generating set for $\cC$ consisting of two fields of weight $n/2$ together with $n-1$ additional fields of weights $1,2,\dots, n-1$. There are no nontrivial normally ordered polynomial relations among these generators and their derivatives, so they freely generate $\cC$. For $n\leq 4$, we show that $\cC$ coincides with the $\cW^{(2)}_n$-algebra at critical level and based on OPE calculations, we conjecture that this holds for all $n$. The appearance of $\cW^{(2)}_n$ in this context is quite mysterious, and a conceptual explanation would be very interesting.

We study the representation theory of $\cC$ via its Zhu algebra, which we identify with the ring $\cD^{SL_n\times SL_n}$. Here $\cD$ denotes the Weyl algebra on the space of $n\times n$ matrices. The generators of $\cD^{SL_n\times SL_n}$ correspond to the generators of $\cC$, and we classify the irreducible, finite dimensional modules over $\cD^{SL_n\times SL_n}$, modulo an explicit formula for certain relations among the generators. These modules are in one-to-one correspondence with the irreducible, admissible $\cC$-modules $M = \bigoplus_{n\geq 0} M_n$ for which each $M_k$ is finite dimensional. For $n\leq 4$, we write down these relations in $\cD^{SL_n\times SL_n}$, so our classification is explicit in these cases. In general, the representation theory of $\cW$-algebras at critical level is an important problem about which very little is known. It would be interesting to see whether any features of our classification of modules for the critical level $\cW^{(2)}_n$-algebra carry over to the case of more general $\cW$-algebras at critical level.

In the case $n=4$, there is an application of our results to physics \cite{CGL}.
Berkovits \cite{Be} introduced a sigma model that conjecturally describes a super Yang-Mills theory.
This model can be formulated as a perturbation of a theory containing two copies of $\cW$.
The perturbation is in terms of the currents of the affine vertex superalgebra $\cA$ of $\gp \gs\gl(4|4)$ at critical level, and the algebra that is preserved by the perturbation is two copies of $\cC = \text{Com}(\cA,\cW)$.

\section{Vertex algebras}
In this section, we define vertex algebras, which have been discussed from various different points of view in the literature (see, e.g., \cite{Bo}\cite{BD}\cite{FLM}\cite{K}\cite{FBZ}). We will follow the formalism developed in \cite{LZ} and partly in \cite{LiI}. Let $V=V_0\oplus V_1$ be a super vector space over $\mathbb{C}$, and let $z,w$ be formal variables. By $\text{QO}(V)$, we mean the space of all linear maps $$V\rightarrow V((z)):=\{\sum_{n\in\mathbb{Z}} v(n) z^{-n-1}|
v(n)\in V,\ v(n)=0\ \text{for} \ n>>0 \}.$$ Each element $a\in \text{QO}(V)$ can be
uniquely represented as a power series
$$a=a(z):=\sum_{n\in\mathbb{Z}}a(n)z^{-n-1}\in \text{End}(V)[[z,z^{-1}]].$$ We
refer to $a(n)$ as the $n$th Fourier mode of $a(z)$. Each $a\in
\text{QO}(V)$ is assumed to be of the shape $a=a_0+a_1$ where $a_i:V_j\ra V_{i+j}((z))$ for $i,j\in\mathbb{Z}/2\mathbb{Z}$, and we write $|a_i| = i$.

On $\text{QO}(V)$ there is a set of nonassociative bilinear operations
$\circ_n$, indexed by $n\in\mathbb{Z}$, which we call the $n$th circle
products. For homogeneous $a,b\in \text{QO}(V)$, they are defined by
$$
a(w)\circ_n b(w)=\text{Res}_z a(z)b(w)~\iota_{|z|>|w|}(z-w)^n-
(-1)^{|a||b|}\text{Res}_z b(w)a(z)~\iota_{|w|>|z|}(z-w)^n.
$$
Here $\iota_{|z|>|w|}f(z,w)\in\mathbb{C}[[z,z^{-1},w,w^{-1}]]$ denotes the
power series expansion of a rational function $f$ in the region
$|z|>|w|$. We usually omit the symbol $\iota_{|z|>|w|}$ and just
write $(z-w)^{-1}$ to mean the expansion in the region $|z|>|w|$,
and write $-(w-z)^{-1}$ to mean the expansion in $|w|>|z|$. It is
easy to check that $a(w)\circ_n b(w)$ above is a well-defined
element of $\text{QO}(V)$.

The nonnegative circle products are connected through the OPE formula.
For $a,b\in \text{QO}(V)$, we have \begin{equation} \label{opeformula} a(z)b(w)=\sum_{n\geq 0}a(w)\circ_n
b(w)~(z-w)^{-n-1}+:a(z)b(w):,\end{equation} which is often written as
$a(z)b(w)\sim\sum_{n\geq 0}a(w)\circ_n b(w)~(z-w)^{-n-1}$, where
$\sim$ means equal modulo the term $$
:a(z)b(w): \ =a(z)_-b(w)\ +\ (-1)^{|a||b|} b(w)a(z)_+.$$ Here
$a(z)_-=\sum_{n<0}a(n)z^{-n-1}$ and $a(z)_+=\sum_{n\geq
0}a(n)z^{-n-1}$. Note that $:a(w)b(w):$ is a well-defined element of
$\text{QO}(V)$. It is called the {\it Wick product} of $a$ and $b$, and it
coincides with $a\circ_{-1}b$. The other negative circle products
are related to this by
$$ n!~a(z)\circ_{-n-1}b(z)=\ :(\partial^n a(z))b(z):\ ,$$
where $\partial$ denotes the formal differentiation operator
$\frac{d}{dz}$. For $a_1(z),\dots ,a_k(z)\in \text{QO}(V)$, the $k$-fold
iterated Wick product is defined to be
$$ :a_1(z)a_2(z)\cdots a_k(z): = :a_1(z)b(z):,$$
where $b(z)= :a_2(z)\cdots a_k(z):$. We often omit the formal variable $z$ when no confusion can arise.

The set $\text{QO}(V)$ is a nonassociative algebra with the operations
$\circ_n$ and a unit $1$. We have $1\circ_n a=\delta_{n,-1}a$ for
all $n$, and $a\circ_n 1=\delta_{n,-1}a$ for $n\geq -1$. A linear subspace $\cA\subset \text{QO}(V)$ containing 1 which is closed under the circle products will be called a {\it quantum operator algebra} (QOA).
In particular, $\cA$ is closed under $\partial$
since $\partial a=a\circ_{-2}1$. Many formal algebraic
notions are immediately clear: a homomorphism is just a linear
map that sends $1$ to $1$ and preserves all circle products; a module over $\cA$ is a
vector space $M$ equipped with a homomorphism $\cA\rightarrow
\text{QO}(M)$, etc. A subset $S=\{a_i|\ i\in I\}$ of $\cA$ is said to generate $\cA$ if any element $a\in\cA$ can be written as a linear
combination of nonassociative words in the letters $a_i$, $\circ_n$, for
$i\in I$ and $n\in\mathbb{Z}$. We say that $S$ {\it strongly generates} $\cA$ if any $a\in\cA$ can be written as a linear combination of words in the letters $a_i$, $\circ_n$ for $n<0$. Equivalently, $\cA$ is spanned by the collection $\{ :\partial^{k_1} a_{i_1}(z)\cdots \partial^{k_m} a_{i_m}(z):| ~i_1,\dots,i_m \in I,~ k_1,\dots,k_m \geq 0\}$. We say that $S$ {\it freely generates} $\cA$ if there are no nontrivial normally ordered polynomial relations among the generators and their derivatives.

We say that $a,b\in \text{QO}(V)$ {\it quantum commute} if $(z-w)^N
[a(z),b(w)]=0$ for some $N\geq 0$. Here $[,]$ denotes the super bracket. This condition implies that $a\circ_n b = 0$ for $n\geq N$, so (\ref{opeformula}) becomes a finite sum. A {\it commutative QOA} (CQOA) is a QOA whose elements pairwise quantum commute. Finally, the notion of a CQOA is equivalent to the notion of a vertex algebra. Every CQOA $\cA$ is itself a faithful $\cA$-module, called the {\it left regular
module}. Define
$$\rho:\cA\rightarrow \text{QO}(\cA),\ \ \ \ a\mapsto\hat a,\ \ \ \ \hat
a(\zeta)b=\sum_{n\in\mathbb{Z}} (a\circ_n b)~\zeta^{-n-1}.$$ Then $\rho$ is an injective QOA homomorphism,
and the quadruple of structures $(\cA,\rho,1,\partial)$ is a vertex
algebra in the sense of Frenkel et al. \cite{FLM}. Conversely, if $(V,Y,{\bf 1},D)$ is
a vertex algebra, the collection $Y(V)\subset \text{QO}(V)$ is a
CQOA. {\it We will refer to a CQOA simply as a
vertex algebra throughout the rest of this paper}.

\begin{example}[Affine vertex algebras] Let $\gg$ be a finite dimensional, complex Lie (super)algebra, equipped with a symmetric, 
invariant bilinear form $B$. The loop algebra $\gg[t,t^{-1}] = \gg\otimes \mathbb{C}[t,t^{-1}]$ has a one-dimensional 
central extension $\hat{\gg} = \gg[t,t^{-1}]\oplus \mathbb{C}\kappa$ determined by $B$, with bracket 
$$[\xi t^n, \eta t^m] = [\xi,\eta] t^{n+m} + n B(\xi,\eta) \delta_{n+m,0} \kappa,$$ and $\mathbb{Z}$-gradation 
$\text{deg}(\xi t^n) = n$, $\text{deg}(\kappa) = 0$. Let $\hat{\gg}_{\geq 0} = \bigoplus_{n\geq 0} \hat{\gg}_n$, where $\hat{\gg}_n$ 
denotes the subspace of degree $n$, and let $C$ be the one-dimensional $\hat{\gg}_{\geq 0}$-module on which $\xi t^n$ acts trivially 
for $n\geq 0$, and $\kappa$ acts by $k$ times the identity. Define $V = U(\hat{\gg})\otimes_{U(\hat{\gg}_{\geq 0})} C$, and 
let $X^{\xi}(n)\in \text{End}(V)$ be the linear operator representing $\xi t^n$ on $V$. 
Define $X^{\xi} (z) = \sum_{n\in\mathbb{Z}} X^{\xi} (n) z^{-n-1}$, which is easily seen to lie in $\text{QO}(V)$ and satisfy the 
OPE relation $$X^{\xi}(z)X^{\eta} (w)\sim kB(\xi,\eta) (z-w)^{-2} + X^{[\xi,\eta]}(w) (z-w)^{-1} .$$ 
The vertex algebra $V_k(\gg,B)$ generated by $\{X^{\xi}| \ \xi \in\gg\}$ is known as the {\it universal affine vertex algebra} associated 
to $\gg$ and $B$ at level $k$.

Note that, for some Lie superalgebras, as, for example, $\gp\gs\gl(n|n)$, the supertrace in the adjoint representation is degenerate.
If the bilinear form $B$ is the standardly normalized supertrace in the adjoint representation of $\gg$ and if $B$ is nondegenerate, then the universal 
affine vertex algebra associated to $\gg$ and $B$ of level $k$ is
denoted by $V_k(\gg)$. 

We recall the \emph{Sugawara construction} for affine vertex superalgebras following \cite{KRW}.
Suppose that $\gg$ is simple and that $B$ is nondegenerate. Let $\{\xi\}$ and $\{\xi'\}$ be dual bases of $\gg$, that is, $B(\xi',\eta)=\delta_{\xi,\eta}$.
Then the Casimir operator is $C_2=\sum_{\xi}\xi\xi'$.
The dual Coxeter number $h^\vee$ with respect to the bilinear form $B$ is one-half the eigenvalue of $C_2$ in the adjoint representation of $\gg$.
If $k+h^\vee\neq0$, there is a Virasoro field
\begin{equation}
L(z) = \frac{1}{2(k+h^\vee)}\sum_\xi :X^{\xi}(z)X^{\xi'}(z):
\end{equation}
of central charge
\begin{equation}
c= \frac{k\text{sdim}\gg}{k+h^\vee}\, .
\end{equation}
This Virasoro element is known as the {\it Sugawara conformal vector}, and each $X^{\xi}$ is primary of weight one. 
At the critical level $k = - h^{\vee}$, $L(z)$ does not exist, but $V_{-h^{\vee}} (\gg, B)$ still possesses a quasi-conformal structure, that is, an action of the Lie subalgebra $\{L_n|~n\geq -1\}$ of the Virasoro algebra such that $L_{-1}$ acts by translation and $L_0$ acts diagonalizably.

\end{example}

\begin{example}[Affine vertex superalgebra of $\gp\gg\gl(n|n)$ at critical level]\label{ex:pgl}
The Lie superalgebra $\gg\gl(n|n)$ has a basis $\{E^\pm_{ab},F^\pm_{ab}|1\leq a,b\leq n\}$.
The $E^\pm_{ab}$ generate two commuting copies of $\gg\gl(n)$ and the $F^\pm_{ab}$ are odd.
The relations are
\begin{equation}
\begin{split}
[E_+^{ab},E_+^{cd}]  &=  \delta^{bc}E_+^{ad}-\delta^{ad}E_+^{cb}\ , \qquad 
[E_-^{ab},E_-^{cd}]  =  \delta^{bc}E_-^{ad}-\delta^{ad}E_-^{cb} \ , \qquad 
[E_+^{ab},E_-^{cd}]  =  0\ , \\
[E_+^{ab},F_+^{cd}]  &=  \delta^{bc}F_+^{ad}\ , \qquad \qquad\quad\ \ \ \, 
[E_-^{ab},F_-^{cd}]  =  \delta^{bc}F_-^{ad} \ , \\
[E_+^{ab},F_-^{cd}]  &=  -\delta^{ad}F_-^{cb}\ , \qquad\qquad\ \ \ \ \, 
[E_-^{ab},F_+^{cd}]  =  -\delta^{ad}F_+^{cb}\ , \\
[F_+^{ab},F_-^{cd}]  &=  \delta^{bc}E_+^{ad}+\delta^{ad}E_-^{cb}\ , \qquad 
[F_+^{ab},F_+^{cd}]  = [F_-^{ab},F_-^{cd}]  = 0\ .
\end{split}
\end{equation}
The element $C=\sum_{a}(E_+^{aa}+E_-^{aa})$ is central. The quotient of $\gg\gl(n|n)$ by the one-dimensional ideal generated by $C$ is the Lie superalgebra
$\gp\gg\gl(n|n)$. This superalgebra is not simple, since the element $K=\sum_{a}(E_+^{aa}-E_-^{aa})$ is not in the derived subalgebra.
Nonetheless, $\gp\gg\gl(n|n)$ possesses a unique Casimir operator, which acts trivially in the adjoint representation.
Hence, the dual Coxeter number is zero. 
The affine vertex superalgebra of $\gp\gg\gl(n|n)$ at critical level $k=0$ is generated by $\{X^{\xi}| \ \xi \in\gp\gg\gl(n|n)\}$.
The operator product algebra is 
\begin{equation}
X^{\xi}(z)X^{\eta} (w)\sim X^{[\xi,\eta]}(w) (z-w)^{-1} .
\end{equation}
\end{example}
 
\begin{example}[$\beta\gamma$ and $bc$ systems] Let $V$ be a finite dimensional complex vector space. The $\beta\gamma$ system or algebra of chiral differential operators $\cS = \cS(V)$ was introduced in \cite{FMS}. It is the unique vertex algebra with even generators $\beta^{x}(z)$, $\gamma^{x'}(z)$ for $x\in V$, $x'\in V^*$, which satisfy
\begin{equation}\beta^x(z)\gamma^{x'}(w)\sim\langle x',x\rangle (z-w)^{-1},\ \ \ \ \ \ \gamma^{x'}(z)\beta^x(w)\sim -\langle x',x\rangle (z-w)^{-1},\end{equation}
$$\beta^x(z)\beta^y(w)\sim 0,\ \ \ \ \ \gamma^{x'}(z)\gamma^{y'}(w)\sim 0.$$ Here $\bra,\ket$ denotes the natural pairing between $V^*$ and $V$. We give $\cS$ the conformal structure 
\begin{equation} \label{virasorobg} L_{\cS}(z) =  \frac{1}{2} \sum_{i=1}^n :\beta^{x_i}\partial\gamma^{x'_i}: -\frac{1}{2} \sum_{i=1}^n :\partial\beta^{x_i}\gamma^{x'_i}:,\end{equation} under which $\beta^{x_i}(z)$ and $\gamma^{x'_i}(z)$ are both primary of weight $1/2$. Here $\{x_1,\dots,x_n\}$ is a basis for $V$ and $\{x'_1,\dots,x'_n\}$ is the dual basis for $V^*$. 

Similarly, the $bc$ system $\cE = \cE(V)$, which was also introduced in \cite{FMS}, is the unique vertex superalgebra with odd generators $b^{x}(z)$, $c^{x'}(z)$ for $x\in V$, $x'\in V^*$, which satisfy
\begin{equation}b^x(z)c^{x'}(w)\sim\langle x',x\rangle (z-w)^{-1},\ \ \ \ \ \ c^{x'}(z)b^x(w)\sim \langle x',x\rangle (z-w)^{-1},\end{equation}
$$b^x(z)b^y(w)\sim 0,\ \ \ \ \ c^{x'}(z)c^{y'}(w)\sim 0.$$ We give $\cE$ the conformal structure \begin{equation} \label{virasorobc} L_{\cE}(z) = \frac{1}{2} \sum_{i=1}^n :\partial b^{x_i} c^{x'_i}: - \frac{1}{2} \sum_{i=1}^n  :b^{x_i} \partial c^{x'_i} :,\end{equation} under which $b^{x_i}(z)$ and $c^{x'_i}(z)$ are primary of weight $1/2$. \end{example}

\begin{example}[The $\cW_n^{(2)}$-algebra at critical level]
In \cite{FS}, Feigin and Semikhatov constructed a family of vertex algebras, the $\cW_n^{(2)}$-algebra at level $k$. Here $k\in\mathbb C$ and $n\geq 2$ is a positive integer. The authors provide two different constructions, as an algebra in the kernel of screenings associated to the quantum super group $\mathcal U_q(\gs\gl(n|1))$ and as a subalgebra of $\text{Com}(V_{k'}(\gs\gl_n)\otimes V(\gg\gl_1),V_{k'}(\gs\gl(n|1))\otimes V_L)$, where $V_L$ is a rank one lattice vertex algebra, and the levels are related by $(k'+n-1)(k+n-1)=1$. 
The $\cW_n^{(2)}$-algebra is generated by two bosonic fields of conformal weight $n/2$, and is expected to be strongly generated by these two fields and $n-1$ additional bosonic fields of conformal weights $1,2,3,\dots,n-1$. The $\cW_2^{(2)}$-algebra at level $k$ is the affine vertex algebra of $\gs\gl_2$ at level $k$ and the $\cW_3^{(2)}$-algebra at level $k$ is the Bershadsky-Polyakov algebra of level $k$, a $\cW$-algebra that is defined as the quantum reduction of $V_{k}(\gs\gl_3)$ for the nonprincipal embedding of $\gs\gl_2$ in $\gs\gl_3$. In general, the $\cW_n^{(2)}$-algebra of level $k$ is believed to be a similar quantum reduction of $V_{k}(\gs\gl_n)$. Recall that, for a given simple Lie (super) algebra $\gg$ and an embedding of $\gs\gl_2$, a strong generating set of fields of the quantum reduced $\cW$-algebra is known \cite{BT,FF,KRW}. We can decompose $\gg$ into irreducible representations of $\gs\gl_2$, and strong generators
are in one-to-one correspondence with lowest-weight states. Moreover, the conformal weights of these generating fields are related to the dimension $d$ of the irreducible lowest-weight representation as $(d+1)/2$.
Returning to the present example, there is the obvious embedding of $\gs\gl_{n-1}$ in $\gs\gl_n$, such that $\gs\gl_n$ decomposes into the adjoint representation plus the trivial representation plus the standard and its conjugate module of $\gs\gl_{n-1}$. Consider the nonprincipal embedding of $\gs\gl_2$ in $\gs\gl_n$, where the $\gs\gl_2$ is principal in the $\gs\gl_{n-1}$-subalgebra. The corresponding quantum reduced $\cW$-algebra is strongly generated by two bosonic fields of conformal weight $n/2$, plus $n-1$ additional bosonic fields of conformal weights $1,2,3,\dots,n-1$. For general $n$ it is an open question whether this quantum reduced $\cW$-algebra is the Feigin-Semikhatov $\cW_n^{(2)}$-algebra.

The $\cW_n^{(2)}$-algebra contains a rank one Heisenberg algebra with generator $H$, and for generic level a Virasoro algebra of central charge 
\begin{equation}\label{centralcharge}
c_n(k) = -\frac{\bigl((k+n)(n-1)-n\bigr)\bigl((k+n)(n-2)n-n^2+1\bigr)}{k+n}\, .
\end{equation}
The operator product algebra is unknown for general $n$. However, Feigin and Semikhatov obtained this algebra for $n=2,3,4$ \cite{FS}. The $\cW_n^{(2)}$-algebra at critical level $k=-n$ contains a large center but no Virasoro algebra. In the case $n=2$, the critical level $\cW_2^{(2)}$-algebra is the affine vertex algebra $V_{-2}(\gs\gl_2)$. For $n=3$, the critical level $\cW_3^{(2)}$-algebra has generators $H,X^\pm, S_2$, and the nonregular operator product algebra is
\begin{equation}\label{W3}
\begin{split}
X^+(z)X^-(w)&\sim 6(z-w)^{-3}-6H(w)(z-w)^{-2}+\\
&\quad \Bigl(3:H(w)H(w):-S_2(w)-3\partial H(w)\Bigr)(z-w)^{-1}\\
H(z)X^\pm(w)&\sim \pm X^\pm(w)(z-w)^{-1} \\
H(z)H(w)&\sim -(z-w)^{-2}\, .
\end{split}
\end{equation}

Similarly, the critical level $\cW_4^{(2)}$-algebra is generated by fields $H,X^\pm, S_2,S_3$, and the nonregular operator product algebra is

\begin{equation}\label{W4}
\begin{split}
X^+(z)X^-(w)&\sim -24(z-w)^{-4}+24H(w)(z-w)^{-3}+\\
&\quad \Bigl(2S_2(w)-12:H(w)H(w):+12\partial H(w)\Bigr)(z-w)^{-2}+\\
&\quad\Bigl(-2S_3(w)+\partial S_2(w)-2:S_2(w)H(w):+4:H(w)H(w)H(w):+\\
&\quad\ \ -12:\partial H(w)H(w):+4\partial^2 H(w)\Bigr)(z-w)^{-1}\\
H(z)X^\pm(w)&\sim \pm X^\pm(w)(z-w)^{-1} \\
H(z)H(w)&\sim -(z-w)^{-2}\, .
\end{split}
\end{equation}
\end{example}

\subsection{The commutant construction}

Let $\cV$ be a vertex algebra, and let $\cA$ be a subalgebra of $\cV$. 
The commutant of $\cA$ in $\cV$, denoted by $\text{Com}(\cA,\cV)$, is the subalgebra of vertex operators $v\in\cV$ such that $[a(z),v(w)] = 0$ 
for all $a\in\cA$. Equivalently, $a(z)\circ_n v(z) = 0$ for all $a\in\cA$ and $n\geq 0$. 
We regard $\text{Com}(\cA,\cV)$ as the algebra of  invariants in $\cV$ under the action of $\cA$. 
If $\cA$ is a homomorphic image of an affine vertex algebra $V_k(\gg, B)$ for some Lie superalgebra $\gg$ and bilinear form $B$, $\text{Com}(\cA,\cV)$ is just the invariant space $\cV^{\gg[t]}$.

\subsection{The Zhu functor}
Let $\cV$ be a vertex algebra with weight grading $\cV = \bigoplus_{n\in\mathbb{Z}} \cV_n$. 
The Zhu functor \cite{Zh} attaches to $\cV$ an associative algebra $A(\cV)$, together with a surjective linear map 
$\pi_{\text{Zhu}}:\cV\ra A(\cV)$. For $a\in \cV_m$, and $b\in\cV$, define 
$$a*b = \text{Res}_z \bigg (a(z) \frac{(z+1)^{m}}{z}b\bigg),$$ and extend $*$ by linearity to a bilinear operation $\cV\otimes \cV\ra \cV$. 
Let $O(\cV)$ denote the subspace of $\cV$ spanned by elements of the form $$a\circ b = \text{Res}_z \bigg (a(z) \frac{(z+1)^{m}}{z^2}b\bigg),$$ 
for $a\in\cV_m$, and let $A(\cV)$ be the quotient $\cV/O(\cV)$, with projection $\pi_{\text{Zhu}}:\cV\ra A(\cV)$. 
For $a,b\in \cV$, $a\sim b$ means $a-b\in O(\cV)$, and $[a]$ denotes the image of $a$ in $A(\cV)$. 
\begin{thm} (Zhu) $O(V)$ is a two-sided ideal in $V$ under the product $*$, and $(A(V),*)$ is an associative algebra with unit $[1]$. 
The assignment $\cV\mapsto A(\cV)$ is functorial.\end{thm}
Let $\mathcal{V}$ be a vertex algebra which is strongly generated by a set of weight-homogeneous elements $\alpha_i$ of weights $w_i$, 
for $i$ in some index set $I$. Then $A(\mathcal{V})$ is generated by $\{ a_i = \pi_{\text{Zhu}}(\alpha_i(z))|~i\in I\}$. 
For example, given a Lie (super)algebra $\gg$ with a bilinear form $B$, it is well known that the Zhu algebra $A(V_k(\gg, B))$ is isomorphic to the universal enveloping algebra $U(\gg)$. Given a finite dimensional vector space $V$, the Zhu algebra of $\cS = \cS(V)$ is isomorphic to the Weyl algebra $\cD = \cD(V)$, which is the associative algebra with generators $x'_i,\frac{\partial}{\partial x'_i}$, and relations $[\frac{\partial}{\partial x'_i}, x'_j] = \delta_{i,j}$. The main application of the Zhu functor is to study the representation theory of $\cV$. A $\mathbb{Z}_{\geq 0}$-graded module $M = \bigoplus_{n\geq 0} M_n$ over $\cV$ is called {\it admissible} if, for every $a\in\cV_m$, $a(n) M_k \subset M_{m+k -n-1}$ for all $n\in\mathbb{Z}$. Given $a\in\cV_m$, the Fourier mode $a(m-1)$ acts on each $M_k$. The subspace $M_0$ is then a module over $A(\cV)$ with action $[a]\mapsto a(m-1) \in \text{End}(M_0)$. In fact, $M\mapsto M_0$ provides a one-to-one correspondence between irreducible, admissible $\cV$-modules, and irreducible $A(\cV)$-modules.

The Zhu functor and the commutant construction interact in the following way: for any subalgebra $\cB\subset \cV$, we have a commutative diagram 
\begin{equation}\label{cdgencase} \begin{array}[c]{ccc}
\text{Com}(\cB,\cV)&\stackrel{}{\hookrightarrow}& \cV  \\
\downarrow\scriptstyle{\pi}&&\downarrow\scriptstyle{\pi_{\text{Zhu}}}\\
\text{Com}(B,A(\cV))&\stackrel{}{\hookrightarrow}& A(\cV)
\end{array} .\end{equation} 
Here $B=\pi_{\text{Zhu}}(\cB)\subset A(\cV)$, and $\text{Com}(B,A(\cV))$ is the ordinary commutant in the theory of associative algebras. The horizontal maps are inclusions, and $\pi$ is the restriction of $\pi_{\text{Zhu}}$ to $\text{Com}(\cB,\cV)$. In general, the map $\pi$ need not be surjective and $A(\text{Com}(\cB,\cV))$ need not coincide with $\text{Com}(B,A(\cV))$. However, as we shall see, both these statements are true in the main examples we consider in this paper.

\section{Graded and filtered structures}
Let $\cR$ be the category of vertex algebras $\cA$ equipped with a $\mathbb{Z}_{\geq 0}$-filtration
\begin{equation} \cA_{(0)}\subset\cA_{(1)}\subset\cA_{(2)}\subset \cdots,\ \ \ \cA = \bigcup_{k\geq 0}
\cA_{(k)}\end{equation} such that $\cA_{(0)} = \mathbb{C}$, and, for all
$a\in \cA_{(k)}$, $b\in\cA_{(l)}$, we have
\begin{equation} \label{goodi} a\circ_n b\in\cA_{(k+l)},\ \ \ \text{for}\
n<0,\end{equation}
\begin{equation} \label{goodii} a\circ_n b\in\cA_{(k+l-1)},\ \ \ \text{for}\
n\geq 0.\end{equation}
Elements $a(z)\in\cA_{(d)}\setminus \cA_{(d-1)}$ are said to have degree $d$.

Filtrations on vertex algebras satisfying (\ref{goodi})-(\ref{goodii})~were introduced in \cite{LiII}, and are known as {\it good increasing filtrations}. Setting $\cA_{(-1)} = \{0\}$, the associated graded object $\text{gr}(\cA) = \bigoplus_{k\geq 0}\cA_{(k)}/\cA_{(k-1)}$ is a
$\mathbb{Z}_{\geq 0}$-graded associative, supercommutative algebra with a
unit $1$ under a product induced by the Wick product on $\cA$. For each $r\geq 1$, we have the projection \begin{equation} \label{projmap}\phi_r: \cA_{(r)} \ra \cA_{(r)}/\cA_{(r-1)}\subset \text{gr}(\cA).\end{equation} 
Moreover, $\text{gr}(\cA)$ has a derivation $\partial$ of degree zero
(induced by the operator $\partial = \frac{d}{dz}$ on $\cA$), and
for each $a\in\cA_{(d)}$ and $n\geq 0$, the operator $a\circ_n$ on $\cA$
induces a derivation of degree $d-k$ on $\text{gr}(\cA)$, which we denote by $a(n)$. Here $$k  = \text{sup} \{ j\geq 1|~ \cA_{(r)}\circ_n \cA_{(s)}\subset \cA_{(r+s-j)}~\forall r,s,n\geq 0\},$$ as in \cite{LL}. Finally, these derivations give $\text{gr}(\cA)$ the structure of a vertex Poisson algebra.

The assignment $\cA\mapsto \text{gr}(\cA)$ is a functor from $\cR$ to the category of $\mathbb{Z}_{\geq 0}$-graded supercommutative rings with a differential $\partial$ of degree 0, which we will call $\partial$-rings. A $\partial$-ring is the same as an {\it abelian} vertex algebra, that is, a vertex algebra $\cV$ in which $[a(z),b(w)] = 0$ for all $a,b\in\cV$. A $\partial$-ring $A$ is said to be generated by a subset $\{a_i|~i\in I\}$ if $\{\partial^k a_i|~i\in I, k\geq 0\}$ generates $A$ as a graded ring. The key feature of $\cR$ is the following reconstruction property \cite{LL}.

\begin{lemma}\label{reconlem}Let $\cA$ be a vertex algebra in $\cR$ and let $\{a_i|~i\in I\}$ be a set of generators for $\text{gr}(\cA)$ as a $\partial$-ring, where $a_i$ is homogeneous of degree $d_i$. If $a_i(z)\in\cA_{(d_i)}$ are vertex operators such that $\phi_{d_i}(a_i(z)) = a_i$, then $\cA$ is strongly generated as a vertex algebra by $\{a_i(z)|~i\in I\}$.\end{lemma}

For example, the $\beta\gamma$ system $\cS = \cS(V)$ has such a filtration, where $\cS_{(r)}$ is defined to be the linear span of \begin{equation} \label{goodsv} \{:\partial^{k_1} \beta^{x_1} \cdots \partial^{k_s}\beta^{x_s}\partial^{l_1} \gamma^{y'_1}\cdots \partial^{l_t}\gamma^{y'_t}:|\ x_i\in V,~y'_i\in V^*,\  k_i,l_i\geq 0,~ s+t \leq r\}.\end{equation} Then $\cS\cong \text{gr}(\cS)$ as linear spaces, and as a commutative algebra, we have \begin{equation} \label{structureofgrs} \text{gr}(\cS)\cong \text{Sym}  \bigoplus_{k\geq 0} (V_k\oplus V^*_k),\ \ \ \ \ V_k = \{\beta^{x}_k |~ x\in V\},\ \ \ \ \ V^*_k = \{\gamma^{x'}_k |\  x'\in V^*\}.\end{equation} Here $\beta^{x}_k$ and $\gamma^{x'}_k$ are the images of $\partial^k \beta^{x}(z)$ and $\partial^k\gamma^{x'}(z)$ in $\text{gr}(\cS)$ under the projection $\phi_1: \cS_{(1)}\ra \cS_{(1)}/\cS_{(0)}\subset \text{gr}(\cS)$. Similarly, $\cE = \cE(V)$ admits such a filtration where $\cE_{(r)}$ is spanned by the iterated Wick products of $b^{x}, c^{x'}$, and their derivatives, of length at most $r$. As above, we have $ \cE \cong \text{gr}(\cE)$ as linear spaces and \begin{equation} \label{structureofgre}\text{gr}(\cE) \cong \bigwedge \bigoplus_{k\geq 0} (V_k \oplus V^*_k)\ \ \ \ \ V_k = \{b^{x}_k |~ x\in V\},\ \ \ \ \ V^*_k = \{c^{x'}_k |\  x'\in V^*\}\end{equation} as supercommutative algebras. These filtrations are important because they allow the description of a commutant to be reduced to a problem in commutative algebra.

\section{A free field realization of $\gp\gg\gl(n|n)$}
For $n\geq 1$, the Lie superalgebra $\gg = \gp\gg\gl(n|n)$ has a grading $$\gg = \gg_{-1} \oplus \gg_0 \oplus \gg_{1},$$ where the odd 
subalgebras $\gg_{-1}$ and $\gg_1$ both have dimension $n^2$, and the even subalgebra $\gg_0 = \gs\gl_n \oplus \gs\gl_n \oplus \gg\gl_1$. 
The corresponding affine vertex superalgebra at critical level $V_0(\gg)$ (example \ref{ex:pgl}) has a free field realization as a subalgebra of the $bc\beta\gamma$ system $\cW= \cE \otimes \cS$ associated to the vector space $V$ of $n\times n$ complex matrices. We work in the basis $x_{ab}$ for $V$ and dual basis $x'_{ab}$ for $V^*$, for $a,b = 1,\dots, n$. We denote the generators $\beta^{x_{ab}}$, $\gamma^{x'_{ab}}$, $b^{x_{ab}}$, $c^{x'_{ab}}$ for $\cW$ by  $\beta^{ab}$, $\gamma^{ab}$, $b^{ab}$, $c^{ab}$, respectively. We use the convention that repeated indices are always summed over. The free field realization of $V_0(\gg)$ is given as follows:
\begin{equation}\nonumber
\begin{split}
F_-^{ab} &= -b_{ba}\,,\\
E_+^{ab} &= B^{ab}_+ - :b^{bc} c^{ac}:\,,\qquad\qquad\ \ \ \text{where}\ \  B^{ab}_+ = -:\beta^{ac} \gamma^{bc}:\,,\\
E_-^{ab} &= B^{ab}_- + :b^{ca} c^{cb}:\,,\qquad\qquad\ \ \ \text{where}\ \  B^{ab}_- = :\gamma^{ca} \beta^{cb}:\,,\\
F_+^{ab} &= -:c^{cb} B_+^{ac}: - :c^{ac} B_-^{cb}: - :b^{cd} c^{cb} c^{ad}:\,.
\end{split}
\end{equation}
Note that $\sum_{a}(E_+^{aa}+E_-^{aa})=0$. The realization is very similar to a realization of the affine vertex superalgebra 
of $\gg\gl(n|n)$ \cite{CR}.  
Also note that the elements $B_{\pm}^{ab}$ generate $V(\gg\gl_1)$ tensored with two commuting copies of $V_{-n}(\gs\gl_n)$, which has critical level, acting on $\cS$. The derived subalgebra $\tilde{\gg} = [\gg,\gg]$ is $\gp\gs\gl(n|n)$, which has a grading  $$\tilde{\gg} = \tilde{\gg}_{-1} \oplus \tilde{\gg}_0 \oplus \tilde{\gg}_{1},$$ where $\tilde{\gg}_{\pm 1} = \gg_{\pm 1}$ and $\tilde{\gg}_0 = \gs\gl_n \oplus \gs\gl_n$. Let $\cA$ be the image of $V_0(\gp\gs\gl(n|n))$ inside $\cW$, and let $\cB$ be the image of $V_{-n}(\gs\gl_n) \otimes V_{-n}(\gs\gl_n)$ inside $\cS$. 

\begin{lemma} The commutant $\cC = \text{Com}(\cA, \cW)$ coincides with $\text{Com}(\cB, \cS)$, and in particular is purely bosonic.
\end{lemma}

\begin{proof} Clearly $\text{Com}(\cB, \cS)\subset \cC$, and to prove the opposite inclusion, it is enough to prove that $\cC\subset \cS$, since $\cS \cap \cC=\text{Com}(\cB,\cS)$. First, $\cW$ is graded by fermionic charge, which is just the eigenvalue of the zero mode of $F = -\sum_{a,b=1}^n :b^{ab} c^{ab}:$. Each $b^{ab}$ and $c^{ab}$ have fermionic charge $-1$ and $1$, respectively. Since $E^{ab}_{\pm}$, $F^{ab}_+$, and $F^{ab}_-$ are homogeneous of fermionic charge $0$, $1$, and $-1$, respectively, $\cC$ is graded by fermionic charge. 

Let $\omega \in \cC$ have homogeneous fermionic charge. Since $\omega$ commutes with $F^{ba}_-  = b^{ab}$ for all $a,b = 1,\dots,n$, $\omega$ does not depend on the vertex operators $c^{ab}$ and their derivatives. Therefore $\omega \in (\bra b \ket \otimes \cS)^{\gg[t]}$ where $\bra b \ket$ is the vertex algebra generated by the $b^{ab}$, and thus has fermionic charge $-r$ for some $r\geq 0$. We need to show that $r=0$, which proves that $\omega \in \cS$.

Recall that $\text{gr}(\cS) = \mathbb{C}[\beta^{ab}_k,\gamma^{ab}_k]$ is graded by degree, where each $\beta^{ab}_k, \gamma^{ab}_k$ has degree $1$. Define an auxiliary gradation on $\text{gr}(\cS)$ called {\it height} as follows: $$\text{ht}(\gamma^{aa}) = a,\ \ \ \ \ \text{ht}(\beta^{a+1,a}) = n+a,$$ and $\text{ht}(\gamma^{ab})=0 = \text{ht}(\beta^{cd})$ for $b\neq a$ and $c\neq d+1$. Given a vertex operator $\alpha  \in \cS_{(d)}\setminus \cS_{(d-1)}$ of degree $d$, define $\text{ht}(\alpha) = \text{ht}( \phi_d(\alpha))$ where $\phi_d: \cS_{(d)} \rightarrow \cS_{(d)} / \cS_{(d-1)}\subset \text{gr}(\cS)$ is the usual projection. Write $\omega = \omega_0 + \omega_1$, where $$\omega_0 = \sum_{CBK} :\partial^{k_1} b^{c_1, b_1}\cdots\partial^{k_r} b^{c_r, b_r} P_{CBK}:,\ \ \ \ \ C = \{c_1,\dots, c_r\},\ \ \ B = \{b_1,\dots, b_r\},\  \ \ \ K = \{k_1,\dots, k_r\},$$ with $\text{deg}(P_{CBK}) = e$ and $\text{ht}(P_{CBK}) = h$. Assume that $\omega_0$ is the "leading term" of $\omega$ in the sense that all terms appearing in $\omega_1$ are of the form $: \partial^{k'_1} b^{c'_1, b'_1}\cdots \partial^{k'_r} b^{c'_r, b'_r} P':$ with either $\text{deg}(P')<e$ or $\text{deg}(P') = e$ and $\text{ht}(P')<h$.

Let $c$ be the largest integer appearing in the list $C$ above, and let $b$ be the largest integer such that $b^{cb}$ or any of its derivatives appears in $\omega_0$. Let $k$ be the largest integer for which $\partial^k b^{cb}$ appears, and write $\omega_0 = :\partial^k b^{cb} (W): + W'$, where $W'$ does not depend on $\partial^k b^{cb}$. Suppose first that $1\leq c<n$.

Let $a = c+1$ and act on $\omega$ by $F_+^{ab} \circ_k$. Recall that $$F_+^{ab} = :c^{cb} \beta^{ad} \gamma^{cd}: - : c^{ac} \gamma^{dc} \beta^{db}: + \alpha,$$ where $\alpha\in\cE$ and can be disregarded, since it cannot raise the degree. 

There will be a term in $F_+^{ab} \circ_k \omega$ of the form $:\beta^{c+1,c} \gamma^{c,c}W:$, which has degree $e+2$ and height $h+n+2c$. It is easy to see that no other terms of the same degree and height can occur in $F^{ab}_+\circ_k \omega$. Moreover, since $\text{gr}(\cS)$ is an integral domain, the image of $:\beta^{c+1,c} \gamma^{c,c}W:$ in $\text{gr}(\cW)$ is nonzero. This contradicts the fact that $\omega \in \cC$.

Next, suppose that $c = n$. Let $a = n$, and act as above by $F^{ab}_+\circ_k$. We see that $F^{ab}\circ_k \omega$ will contain the term $:\beta^{n,n-1} \gamma^{n,n-1} W:$, which has degree $e+2$ and height $h+2n-1$ and cannot be canceled by any other term. As above, this contradicts $\omega \in\cC$. It follows that $r=0$ and $\omega \in \cS$ as claimed. \end{proof}

\section{Commutants inside the $\beta\gamma$ system}
 
We are thus led to the problem of computing $\text{Com}(\cB, \cS)$, which is exactly the type of commutant problem considered in \cite{LSSII}. Let $G$ be a connected, reductive group with Lie algebra $\gg$, and let $V$ be a finite dimensional representation of $V$. The induced map $\rho: \gg\rightarrow \text{End}(V)$ induces a vertex algebra homomorphism $\hat{\tau}: V_1(\gg,B) \ra \cS =  \cS(V)$ given by
\begin{equation} \label{deftheta} \hat{\tau}(X^{\xi}(z)) = \theta^{\xi}(z) = - \sum_{i=1}^n~:\gamma^{x'_i}(z)\beta^{\rho(\xi)(x_i)}(z):~.\end{equation}
Here $B$ is the bilinear form $B(\xi,\eta) = -\text{tr}(\rho(\xi)\rho(\eta))$, and  $\{x_1,\dots, x_n\}$ is a basis for $V$, with dual basis $\{x'_1,\dots, x'_n\}$ for $V^*$. Let $\Theta$ be the vertex algebra generated by $\theta^{\xi}, \xi\in\gg$. The commutant $\text{Com}(\Theta, \cS)$, which coincides with $\cS^{\gg[t]}$, will be called the algebra of {\it invariant chiral differential operators} on $V$. It is analogous to the classical ring $\cD^G$ of invariant differential operators. In this notation, $\cD = \cD(V)$ is the Weyl algebra with generators $x'_i, \frac{\partial}{\partial x'_i}$ satisfying $[\frac{\partial}{\partial x'_i},x'_j] = \delta_{i,j}$. 
Equip $\cD$ with the Bernstein filtration \begin{equation}\label{bernstein}\cD_{(0)}\subset \cD_{(1)}\subset  \cdots,\end{equation} defined by $(x'_1)^{k_1} \cdots (x'_n)^{k_n} (\frac{\partial}{\partial x'_1})^{l_1}\cdots (\frac{\partial}{\partial x'_n})^{l_n} \in \cD_{(r)}$ if $k_1 + \cdots +k_n + l_1 + \cdots +l_n \leq r$. Given $\omega\in \cD_{(r)}$ and $\nu\in\cD_{(s)}$, $[\omega,\nu]\in \cD_{(r+s-2)}$, so that \begin{equation} \label{isodi}\text{gr}(\cD) = \bigoplus_{r>0} \cD_{(r)} / \cD_{(r-1)} \cong \text{Sym} (V\oplus V^*).\end{equation} We say that $\text{deg}(\alpha) = d$ if $\alpha\in\cD_{(d)} \setminus \cD_{(d-1)}$. 

The action of $G$ on $V$ induces a Lie algebra homomorphism \begin{equation} \label{defoftau} \tau:\gg\ra \cD,\ \ \ \ \  \xi\mapsto - \sum_{i=1}^n x'_i \rho(\xi)\big(\frac{\partial}{\partial x'_i}\big),\end{equation} which is analogous to (\ref{deftheta}). Given $\xi\in\gg$, $\tau(\xi)$ is just the vector field on $V$ generated by $\xi$, and $\xi$ acts on $\cD$ by $[\tau(\xi),-]$. We can extend $\tau$ to a map $U(\gg)\ra \cD$, and $\cD^G=\text{Com}(\tau(U(\gg)),\cD)$ since $G$ is connected. Moreover, $G$ preserves the filtration on $\cD$, so (\ref{bernstein}) restricts to a filtration
$\cD^{G}_{(0)}\subset \cD^{G}_{(1)} \subset \cdots$ on $\cD^{G}$, and $$\text{gr}(\cD^{G}) \cong \text{gr}(\cD)^{G} \cong \text{Sym} (V\oplus V^*)^{G}.$$

Recall that $V_1(\gg,B)$ and $\cS$ are related via the Zhu functor to $U(\gg)$ and $\cD$, respectively. We have commutative diagrams
\begin{equation}\label{commdiagfirst} \begin{array}[c]{ccc}
V_1(\gg,B) &\stackrel{}{\rightarrow}& \cS  \\
\downarrow\scriptstyle{\pi_{\text{Zhu}} }&&\downarrow\scriptstyle{\pi_{\text{Zhu}}}\\
U(\gg) &\stackrel{}{\rightarrow}& \cD
\end{array},\ \ \ \ \ \ \ \ \ \ \   \begin{array}[c]{ccc}
\cS^{\gg[t]}&\stackrel{}{\hookrightarrow}& \cS  \\
\downarrow\scriptstyle{\pi}&&\downarrow\scriptstyle{\pi_{\text{Zhu}}}\\
\cD^{G} &\stackrel{}{\hookrightarrow}& \cD
\end{array}.\end{equation} The map $V_1(\gg,B) \ra \cS$ is $\hat{\tau}$, and the map $U(\gg) \ra \cD$ coincides with $\tau$ up to modification by a scalar (i.e., an element of degree zero in $\cD$) for central elements of $\gg$. The map $\pi$ is just the restriction of the Zhu map on $\cS$.

\section{Jet schemes}
There is a well-known connection between vertex algebras and jet schemes: for any affine variety $X$, the ring of polynomial functions on the infinite jet scheme or arc space $X_{\infty}$, has the structure of an abelian vertex algebra \cite{BD}. Conversely, the $\partial$-ring $\text{gr}(\cA)$ of a vertex algebra $\cA\in \cR$ can often be realized as the ring of polynomial functions $\cO(X_{\infty})$ for some $X$. In our main example, $\text{gr}(\cS)$ is isomorphic to $\cO((V\oplus V^*)_{\infty})$, and $\text{gr}(\cS^{\gg[t]})\cong \cO(X_{\infty})$ where $X$ is the categorical quotient $(V\oplus V^*)/\!\!/G = \text{Spec}(\cO(V\oplus V^*)^G)$. More generally, whenever $\text{Spec}(\text{gr}(\cA))$ can be realized as $X_{\infty}$ for some $X$, the geometry of $X_{\infty}$ encodes information about the vertex algebra structure of $\cA$.

First, we recall some basic facts about jet schemes, following the notation in \cite{Mu}. Let $X$ be an irreducible scheme of finite type over $\mathbb{C}$. For each integer $m\geq 0$, the $m$th jet scheme $X_m$ is determined by its functor of points: for every $\mathbb{C}$-algebra $A$, we have a bijection
$$\text{Hom}(\text{Spec} (A), X_m) \cong \text{Hom}(\text{Spec} (A[t]/\langle t^{m+1}\rangle ), X).$$ Thus the $\mathbb{C}$-valued points of $X_m$ correspond to the $\mathbb{C}[t]/\langle t^{m+1}\rangle$-valued points of $X$. If $m>p$, we have projections $\pi_{m,p}: X_m \rightarrow X_p$ which are compatible when defined: $\pi_{m,p} \circ \pi_{q,m} = \pi_{q,p}$. Clearly $X_0 = X$ and $X_1$ is the total tangent space $\text{Spec}(\text{Sym} (\Omega_{X/\mathbb{C}}))$. The assignment $X\mapsto X_m$ is functorial, and a morphism $f:X\ra Y$ induces $f_m: X_m \ra Y_m$ for all $m\geq 1$. If $X$ is nonsingular, $X_m$ is irreducible and nonsingular for all $m$. Moreover, if $X,Y$ are nonsingular and $f:Y\ra X$ is a smooth surjection, $f_m$ is surjective for all $m$. 

If $X=\text{Spec}(R)$ where $R= \mathbb{C}[y_1,\dots,y_r] / \langle f_1,\dots, f_k\rangle$, we can find explicit equations for $X_m$. Define new variables $y_j^{(i)}$ for $i=0,\dots, m$, and define a derivation $D$ by $D(y_j^{(i)}) = y_j^{(i+1)}$ for $i<m$, and $D(y_j^{(m)}) =0$. This specifies the action of $D$ on all of $\mathbb{C}[y_1^{(0)},\dots, y_r^{(m)}]$; in particular, $f_{l}^{(i)} = D^i ( f_{l})$ is a well-defined polynomial in  $\mathbb{C}[y_1^{(0)},\dots, y_r^{(m)}]$. Letting $R_m = \mathbb{C}[y_1^{(0)},\dots, y_r^{(m)}] / \langle f_1^{(0)},\dots, f_k^{(m)}\rangle$, we have $X_m \cong \text{Spec} (R_m)$. By identifying $y_j$ with $y_j^{(0)}$, we see that $R$ is naturally a subalgebra of $R_m$. 

The {\it arc space} of $X$ is defined to be $X_{\infty} = \lim_{\infty \leftarrow m} X_m$. If $X = \text{Spec}(R)$ as above, $X_{\infty}=\text{Spec}(R_{\infty})$ where $R_{\infty}  = \mathbb{C}[y_1^{(0)},\dots,y_j^{(i)},\dots] / \bra f_1^{(0)},\dots, f_l^{(i)},\dots\ket$. Here $i=0,1,2,\dots$ and $D (y^{(i)}_j) = y^{(i+1)}_j$ for all $i$. By a theorem of Kolchin \cite{Kol}, $X_{\infty}$ is irreducible whenever $X$ is irreducible.

Let $G$ be a connected, reductive complex algebraic group with Lie algebra $\gg$. For $m\geq 1$, $G_m$ is an algebraic group which is the semidirect product of $G$ with a unipotent group $U_m$. The Lie algebra of $G_m$ is $\gg[t]/t^{m+1}$. Given a linear representation $V$ of $G$, there is an action of $G$ on $\cO(V)$ by automorphisms, and a compatible action of $\gg$ on $\cO(V)$ by derivations, satisfying $$\frac{d}{dt} \text{exp} (t\xi) (f)|_{t=0} = \xi(f),\ \ \ \ \xi\in\gg,\ \ \ \ f\in \cO(V).$$ Choose a basis $\{x_1,\dots,x_n\}$ for $V^*$, so that $$\cO(V) \cong  \mathbb{C}[x_1,\dots,x_n],\ \ \ \ \ \ \cO(V_m) =  \mathbb{C}[x_1^{(i)},\dots,x_n^{(i)}|\ ,0\leq i\leq m].$$ Then $G_m$ acts on $V_m$, and the induced action of $\gg[t]/t^{m+1}$ by derivations on $\cO(V_m)$ is defined on generators by \begin{equation}\label{jetaction}\xi t^r (x_j^{(i)}) = \lambda^r_i \xi(x_j)^{(i-r)}, \ \ \ \ \ \ \lambda^r_i = \bigg\{ \begin{matrix} \frac{i!}{(i-r)!}  & 0\leq r\leq i \cr 0 & r>i \end{matrix}.\end{equation} Via the projection $\gg[t]\ra \gg[t]/t^{m+1}$, $\gg[t]$ acts on $\cO(V_m)$, and the invariant rings $\cO(V_m)^{\gg[t]}$ and $\cO(V_m)^{\gg[t]/t^{m+1}}$ coincide.

The relevance of these invariant rings to our vertex algebra commutant problem is as follows. The map (\ref{deftheta}) induces an action of $\gg[t]$ on $\text{gr}(\cS)$ by derivations of degree zero, defined on generators by 
\begin{equation}\label{actiontheta}\xi t^r(\beta^x_i) = \lambda^r_i \beta^{\rho(\xi)(x)}_{i-r},\ \ \ \ \ \ \xi t^r(\gamma^{x'}_i) = \lambda^r_i \gamma^{\rho^*(\xi)(x')}_{i-r}.\end{equation}
The derivation $\partial$ on $\text{gr}(\cS)$ is given by \begin{equation}\label{actionofpartial}\partial \beta^{x}_i = \beta^x_{i+1},\ \ \ \ \ \ \partial \gamma^{x'}_i = \gamma^{x'}_{i+1}.\end{equation} There is an injective map of $\partial$-rings \begin{equation}\label{injgamm} \text{gr}(\cS^{\gg[t]})\hookrightarrow \text{gr}(\cS)^{\gg[t]},\end{equation} which is in general not surjective. Let $R$ denote the image of \eqref{injgamm}, and suppose that $\{a_i| i\in I\}$ is a collection of generators for $R$ as a $\partial$-ring. By Lemma \ref{reconlem}, any collection of vertex operators $\{a_i(z)\in \cS^{\gg[t]}|~i\in I\}$ such that $d_i = \text{deg}(a_i)$ and $\phi_{d_i}(a_i(z)) = a_i$, is a strong generating set for $\cS^{\gg[t]}$ as a vertex algebra.

By (\ref{actiontheta}) and (\ref{jetaction}) the map $\Phi: \text{gr}(\cS) \ra \cO((V\oplus V^*)_{\infty})$ defined on generators by $$\beta^{x}_k \mapsto x^{(k)},\ \ \ \ \ \gamma^{x'}_k \mapsto (x')^{(k)},$$ is an isomorphism of $\gg[t]$-algebras. Moreover, $\Phi^{-1} D\circ \Phi = \partial$, so we have an isomorphism of differential graded algebras $\text{gr}(\cS)^{\gg[t]} \cong  \cO((V\oplus V^*)_{\infty})^{\gg[t]}$. Since $G$ is connected, $\cO((V\oplus V^*)_{\infty})^{\gg[t]} = \cO((V\oplus V^*)_{\infty})^{G_{\infty}}$, and so we obtain the following lemma.
\begin{lemma} $\text{gr}(\cS)^{\gg[t]} \cong \cO((V\oplus V^*)_{\infty})^{G_{\infty}}$ as differential graded algebras. \end{lemma}
 In general, it is a very subtle problem to find generators for rings of the form $\cO((U)_{\infty})^{G_{\infty}}$ for a $G$-representation $U$. There is a natural map \begin{equation} \label{arcmap} \cO((U/\!\!/G)_{\infty}) \rightarrow \cO(U_{\infty})^{G_{\infty}}\end{equation} which is not an isomorphism in general. We call $U$ {\it stable} if the general $G$-orbit is closed in $U$, and we call $U$ {\it coregular} if $U/\!\!/G$ is smooth, equivalently; $\cO(V)^G$ is a polynomial ring. The following result appears in \cite{LSSI}.
 
\begin{thm} \label{jetthm} Let $G$ be a connected, reductive group, and let $U$ be a $G$-representation which is stable and coregular, such that $\cO(U)$ contains no nontrivial one-dimensional $G$-invariant subspaces. Then \eqref{arcmap} is an isomorphism. In particular, if $\cO(U)^G$ is generated by polynomials $f_1,\dots, f_k$, $\cO(U_{\infty})^{G_{\infty}}$ is generated by $D^i(f_1),\dots,D^i(f_k)$, for $i\geq 0$.
\end{thm}

\section{The main example}

We are interested in $\text{Com}(\Theta, \cS) = \cS^{\gg[t]}$ in the case where $\Theta$ is the image of $V_{-n} (\gs\gl_n) \otimes V_{-n} (\gs\gl_n)$ in $\cS = \cS(V)$, and $V$ is the space of $n\times n$ matrices. First, we consider the corresponding problem involving invariant differential operators. For $j=1,\dots, n$, let $W_j = \mathbb{C}^n$, with basis $\{x_{1j},\dots, x_{nj}\}$, so that $V = \bigoplus_{j=1}^n W_j$. The Weyl algebra $\cD = \cD(V)$ has generators $x'_{ij}, \frac{\partial}{\partial x'_{ij}}$, for $i,j=1,\dots,n$. The left and right actions of $GL_n$ on $V$ induce algebra homomorphisms \begin{equation} \label{taugln} \tau: U (\gg\gl_n)\ra \cD,\ \ \ \ \  \tau': U (\gg\gl_n) \ra \cD\end{equation} such that $\xi\in \gg\gl_n$ and $\eta\in \gg\gl_n$ act on $\cD$ by $[\tau(\xi),-]$ and $[\tau'(\eta),-]$, respectively. It is well known \cite{GW} that, for the left action of $GL_n$, $\cD^{GL_n} = \tau'(U (\gg\gl_n))$, and for the right action of $GL_n$, $\cD^{GL_n} = \tau(U( \gg\gl_n))$. Moreover, $\tau(U (\gg\gl_m))$ and $\tau'(U (\gg\gl_n))$ form a pair of mutual commutants inside $\cD$. It follows that $$\cD^{GL_n\times GL_n} = \tau(U (\gg\gl_n)) \cap \tau'(U (\gg\gl_n)),$$ which is isomorphic to the center $Z(\gg\gl_n) \subset U(\gg\gl_n)$, and is just the polynomial algebra $\mathbb{C}[c_1,\dots,c_n]$. Here $c_i = \tau(\zeta_i)$, where $\zeta_i$ is the $i$th Casimir in $Z(\gg\gl_n)$. Note that $c_i$ has degree $2i$ in the Bernstein filtration and $\tau(Z(\gg\gl_n)) = \tau'(Z(\gg\gl_n))$.

For $n\geq 2$, $\cD^{SL_n \times SL_n}$ is generated by $c_1,\dots, c_n$ together with the $n\times n$ determinants $d = \det[\frac{\partial}{\partial{x'_{ij}}}]$ and $d' = \det[x'_{ij}]$. This follows from Weyl's first fundamental theorem of invariant theory for the standard representation of $SL_n$ \cite{We}, which shows that $\cD^{SL_n}$ is generated by $\cD^{GL_n} \cong U(\gg\gl_n)$ together with $d,d'$, and the fact that $Z(\gg\gl_n)\cong \mathbb{C}[c_1,\dots,c_n]$. Moreover, $c_2,\dots,c_n$ lie in the center of $\cD^{SL_n \times SL_n}$. Let $\tilde{d}, \tilde{d}', \tilde{c}_1\dots \tilde{c}_n$ denote the corresponding elements of $\text{gr}(\cD^{SL_n\times SL_n}) \cong \cO(V\oplus V^*)^{SL_n\times SL_n}$. The ideal of relations among $\tilde{d}, \tilde{d}', \tilde{c}_1\dots \tilde{c}_n$ is generated by a single relation of degree $2n$ of the form \begin{equation} \label{gradedrel} \tilde{d} \tilde{d}' +p(\tilde{c}_1\dots \tilde{c}_n)=0,\end{equation} for some polynomial $p$. We can write down a formula for $p$ as follows. Recall that $\tilde{c}_k$ is the trace of $X^k$, where $X$ is the matrix whose $ij$th entry is $\tau(\xi_{ij})$. Clearly $\det(X) = \tilde{d} \tilde{d}'$. Moreover, the determinant of any $n\times n$ matrix
$X$ can be expressed in terms of traces of powers of $X$.
This formula is called the Newton-Girard formula. 
Let $E_0(X)=1$ and, for $1\leq m\leq n$, define recursively
\begin{equation}
E_m(X) \ = \ -\frac{(-1)^m}{m} \sum_{k=1}^m \text{tr}(X^k)E_{m-k}(X)\, .
\end{equation}
Then $\text{det}(X)=E_n(X)$.
In particular, $p$ contains a nontrivial multiple of $\tilde{c}_n$, so $\tilde{c}_n$ can be expressed as a polynomial in $\tilde{d}, \tilde{d}', \tilde{c}_1\dots \tilde{c}_{n-1}$. After eliminating $\tilde{c}_n$, there are no relations among $\tilde{d}, \tilde{d}', \tilde{c}_1\dots \tilde{c}_{n-1}$, so $\text{gr}(\cD)^{SL_n\times SL_n} = \cO(V\oplus V^*)^{SL_n\times SL_n}$ is a polynomial algebra. 

\begin{lemma} \label{eckapp} In the case where $V$ is the space of $n\times n$ matrices and $G = SL_n\times SL_n$, we have $$\cO((V\oplus V^*)_{\infty})^{\gg[t]} = \cO(((V\oplus V^*)/\!\!/G)_{\infty}).$$ \end{lemma} 

\begin{proof} It is known that any $G$-representation of the form $V\oplus V^*$ is stable. Since $G$ is semisimple, $\cO(V\oplus V^*)$ contains no nontrivial one-dimensional $G$-invariant subspaces. Since $\cO(V\oplus V^*)^G$ is a polynomial algebra, the claim follows from Theorem \ref{jetthm}. \end{proof}

At the vertex algebra level, we obtain maps \begin{equation} \label{embedleft} \hat{\tau}: V_{-n} (\gs\gl_n) \ra \cS,\ \ \ \ \ \ \hat{\tau}' :V_{-n}(\gs\gl_n)\ra \cS \end{equation} corresponding to the left and right actions of $SL_n$ on $V$. In order to study $\text{Com}(\Theta,\cS) = \cS^{\gg[t]}$ for $\gg = \gs\gl_n \oplus \gs\gl_n$, we first consider the structure of $\text{gr}(\cS)^{\gg[t]}$. By Lemma \ref{eckapp}, $\cO((V\oplus V^*)_{\infty})^{\gg[t]}$ is generated as a differential algebra by $\cO(V\oplus V^*)^G$, and since $\cO(V\oplus V^*)^G$ is the polynomial algebra with generators $\tilde{d},\tilde{d}',\tilde{c}_1,\dots, \tilde{c}_{n-1}$, it follows that $\cO((V\oplus V^*)_{\infty})^{\gg[t]}$ is the polynomial algebra with generators $\partial^k \tilde{d}, \partial^k \tilde{d}', \partial^k \tilde{c}_1,\dots, \partial^k \tilde{c}_{n-1}$, for all $k\geq 0$. Under the isomorphism $\cO((V\oplus V^*)_{\infty})^{\gg[t]} \cong \text{gr}(\cS)^{\gg[t]}$, the generators $\tilde{d},\tilde{d}',\tilde{c}_1,\dots, \tilde{c}_{n-1}$ correspond to generators for  $\text{gr}(\cS)^{\gg[t]}$ as a differential algebra, which we also denote by $\tilde{d},\tilde{d}',\tilde{c}_1,\dots, \tilde{c}_{n-1}$.

\begin{lemma} In the case where $V$ is the space of $n\times n$ matrices and $G = SL_n\times SL_n$, the map \eqref{injgamm} is surjective, and is therefore an isomorphism.
\end{lemma}
\begin{proof} We need to find vertex operators $D, D'$ of weight $n/2$ and $C_1,\dots, C_{n-1}$ of weights $1,\dots, n-1$ in $\cS^{\gg[t]}$, such that \begin{enumerate} 
\item $D,D'$ lie in $\cS^{\gg[t]}_{(n)}$, and $\phi_{n} (D) = \tilde{d}$, $\phi_n (D') = \tilde{d}'$.
\item For $i=1,\dots,n-1$, $C_i \in \cS^{\gg[t]}_{(2i)}$, and $\phi_{2i}(C_i) = \tilde{c}_i$.
\end{enumerate} In this notation, $\phi_r: \cS^{\gg[t]}_{(r)} \rightarrow \cS^{\gg[t]}_{(r)} / \cS^{\gg[t]}_{(r-1)}\subset \text{gr}(\cS^{\gg[t]})$ is the projection (\ref{projmap}). Recall that $d'\in \cD^G$ is the $n\times n$ determinant of the matrix whose entries are the linear functions $x'_{ij}$ on $V$. The corresponding element of $\text{gr}(\cS)^{\gg[t]}$ is the $n\times n$ determinant of the matrix whose entries are $\gamma^{x'_{ij}}_0$. Letting $D'$ be the vertex operator in $\cS$ obtained from $d'$ by replacing $\gamma^{x'_{ij}}_0$ with the vertex operator $\gamma^{x'_{ij}}$, and replacing all products with iterated Wick products, it is immediate that $D'$ is $G$-invariant. Moreover, $D'$ is $\gg[t]$-invariant because it only depends on the $\gamma^{x'_{ij}}$, and therefore can have no double contractions with the fields $\theta^{\xi}$ for $\xi\in\gg$. Therefore, $D'$ lies in $\cS^{\gg[t]}$, and the fact that $\phi_n(D') =\tilde{d}'$ is clear by construction. The vertex operator $D$ is defined in the same way with $\beta^{x_{ij}}$ playing the role of $\gamma^{x'_{ij}}$, and the same argument shows that $D\in\cS^{\gg[t]}$ and $\phi_n(D) = d$. We define $$C_1= \sum_{i,j=1}^n :\beta^{x_{ij}} \gamma^{x'_{ij}}:,$$ which is easily seen to have the desired properties.
Finally, since $V_{-n}(\gs\gl_n)$ has critical level, the center of $V_{-n}(\gs\gl_n)$ has generators $\tilde{C}_2,\dots,\tilde{C}_n$ corresponding to the center of $U(\gs\gl_n)$. We define $C_i = \hat{\tau}(\tilde{C}_i)$, which clearly lies in $\cS^{\gg[t]}$ and satisfies $\phi_{2i}(C_i) = \tilde{c}_i$. The existence of these elements of $\cS^{\gg[t]}$ implies that \eqref{injgamm} is surjective, as claimed.  \end{proof}

It follows from Lemma \ref{reconlem} that $\{D,D',C_1,\dots, C_{n-1}\}$ is a strong generating set for $\cS^{\gg[t]}$. Since $C_2,\dots,C_{n-1}$ are normally ordered polynomials in the fields $\theta^{\xi}$ for $\xi\in \gg$, they lie in the center of $\cS^{\gg[t]}$. This is analogous to the fact that $c_2,\dots,c_{n-1}$ lie in the center of $\cD^G$. Moreover, since $\text{gr}(\cS)^{\gg[t]}$ is a polynomial algebra with generators $\tilde{d},\tilde{d}',\tilde{c}_1,\dots,\tilde{c}_{n-1}$ and their derivatives, there are no normally ordered polynomial relations among the vertex operators $D,D',C_1,\dots, C_{n-1}$ and their derivatives. Thus, we have proved the following theorem.

\begin{thm} For all $n\geq 2$, $\cS^{\gg[t]}$ is freely generated as a vertex algebra by $\{D,D',C_1,\dots, C_{n-1}\}$.
\end{thm}

We write down all nontrivial OPE relations in $\cS^{\gg[t]}$ in the cases $n=2,3,4$ for the sake of concreteness.
For general $n$, we are able to compute the leading terms.
\begin{enumerate}
\item In the case $n=2$, the generators are $C_1, D, D'$, which satisfy the OPE relations
\begin{equation}\nonumber
\begin{split}
C_1(z) C_1(w) &\sim -4 (z-w)^{-2}\,,\\
C_1(z) D(w) &\sim - 2 D(w) (z-w)^{-1}\,,\\
C_1(z) D'(w) &\sim 2 D'(w) (z-w)^{-1}\,,\\
D(z) D'(w) &\sim 2 (z-w)^{-2} + C_1(w)(z-w)^{-1}\,.
\end{split}
\end{equation}
Thus, $\cS^{\gg[t]}$ is precisely the critical level affine vertex algebra $V_{-2}(\gs\gl_2)$.

\item In the case $n=3$, the generators are $C_1, C_2, D, D'$, with nontrivial OPE relations
\begin{equation}\nonumber
\begin{split}
C_1(z) C_1(w) &\sim -9 (z-w)^2\,,\\
C_1(z) D(w) &\sim -3 D(w) (z-w)^{-1}\,,\\
C_1(z) D'(w) &\sim 3 D'(w) (z-w)^{-1}\,,\\
D(z) D'(w) &\sim 6 (z-w)^{-3} + 2C_1(w) (z-w)^{-2} +\\
&\quad +\bigg(\frac{1}{3} :C_1(w)C_1(w): -\frac{1}{2} C_2(w) +
\partial C_1(w)\bigg) (z-w)^{-1}\,.
\end{split}
\end{equation}
We see that these are precisely the relations of the critical level $\cW_3^{(2)}$-algebra \eqref{W3}.
\item In the case $n=4$, the generators are $C_1, C_2,C_3, D, D'$, with nontrivial OPE relations
\begin{equation}\nonumber
\begin{split}
C_1(z) C_1(w) &\sim -16 (z-w)^{-1}\,,\\
C_1(z) D(w) &\sim -4 D(w) (z-w)^{-1}\,,\\
C_1(z) D'(w) &\sim 4 D'(w) (z-w)^{-1}\,,\\
D(z) D'(w) &\sim 24 (z-w)^{-4} + 6 C_1(w) (z-w)^{-3}  + \\
&\quad +\bigg(-C_2(w) + \frac{3}{4} :C_1(w) C_1(w): +
3 \partial C_1(w) \bigg) (z-w)^{-2}+\\
&\quad + \bigg(  -\frac{1}{8} C_3(w) -\frac{1}{4} :C_2(w) C_1(w): + \frac{1}{16} :C_1(w) C_1(w) C_1(w): +\\
&\quad +\frac{3}{4} :C_1(w)
\partial C_1(w): + \partial^2 C_1(w) \bigg)(z-w)^{-1}\, .
\end{split}
\end{equation}
We see that these are precisely the relations of the critical level $\cW_4^{(2)}$-algebra \eqref{W4}.
\item For general $n$, the generators are $C_1, \dots ,C_{n-1}, D, D'$, with nontrivial OPE relations
\begin{equation}\nonumber
\begin{split}
C_1(z) C_1(w) &\sim -n^2 (z-w)^{-1}\,,\\
C_1(z) D(w) &\sim -n D(w) (z-w)^{-1}\,,\\
C_1(z) D'(w) &\sim n D'(w) (z-w)^{-1}\,,\\
D(z) D'(w) &\sim n! (z-w)^{-n} + (n-1)! C_1(w) (z-w)^{-(n-1)}  + ... \, .
\end{split}
\end{equation}
We remark that also in \cite{FS}, only the leading OPEs of the $\cW_n^{(2)}$-algebra for $n>4$ were computed.
At critical level they agree with our computations. This leads us to the following conjecture.
\end{enumerate}
\begin{conj}
The commutant $ \cC = \text{Com}(\cA,\cW)$ is isomorphic to the $\cW_n^{(2)}$-algebra at critical level.
\end{conj}

\begin{remark} The center $Z(\cC)$ is freely generated by $\{C_2,\dots,C_n\}$, where $C_n = \hat{\tau}(\tilde{C}_n)$ can be expressed as a normally ordered polynomial in $\{D,D',C_1,\dots, C_{n-1}\}$, and their derivatives. In particular, $Z(\cC)$ is isomorphic to the center $Z(V_{-n}(\gs\gl_n))$. Assuming that $\cW^{(2)}_n$ can be realized as a quantum reduction of $V_k(\gs\gl_n)$, this is consistent with Arakawa's result \cite{A} that, at critical level $k = -h^{\vee}$, $Z(V_k(\gg))$ is isomorphic to the center $Z(\cW_k(\gg,f))$ for any simple $\gg$ and any nilpotent element $f \in\gg$.
\end{remark}

\section{The representation theory of $\cC$}
Given a representation $V$ of a connected, reductive group $G$, recall that we have a commutative diagram $$\begin{array}[c]{ccc}\cS^{\gg[t]} &\stackrel{}{\hookrightarrow}& \cS  \\ \downarrow\scriptstyle{\pi}&&\downarrow\scriptstyle{\pi_{\text{Zhu}}}\\ \cD^{G} &\stackrel{}{\hookrightarrow}& \cD \end{array},$$ where $\pi$ is the restriction of $\pi_{\text{Zhu}}: \cS(V)\rightarrow \cD(V)$. 

\begin{thm} \label{zhudescription} Suppose that $\cO((V\oplus V^*)_{\infty})^{G_{\infty}} \cong\cO(((V\oplus V^*) /\!\!/ G)_{\infty})$ and that the map (\ref{injgamm}) is surjective. Then the Zhu algebra $A(\cS^{\gg[t]})$ is isomorphic to $\cD^G$ and the map $\pi$ above is surjective. \end{thm}

\begin{proof} Fix a generating set $\{f_1,\dots, f_r\}$ for $\cD^G$ such that $f_i \in \cD^G_{(d_i)} \setminus \cD^G_{(d_i-1)}$, which under the above hypotheses corresponds to a strong generating set $\{f_1(z),\dots, f_r(z)\}$ for $\cS^{\gg[t]}$ as a vertex algebra. Without loss of generality, we may assume that $f_i(z) \in \cS^{\gg[t]}_{(d_i)} \setminus  \cS^{\gg[t]}_{(d_i-1)}$. Then $A(\cS^{\gg[t]})$ is generated by $\{\tilde{f}_1,\dots, \tilde{f}_r\}$ where $\tilde{f}_i = \pi_{\text{Zhu}}(f_i(z))$. 

Clearly, $\pi(f_i(z)) = f_i$ up to corrections of lower degree in the Bernstein filtration, and so, by induction on degree, we see that $\pi$ is surjective. The inclusion $\cS^{\gg[t]} \hookrightarrow \cS$ induces a map of Zhu algebras $h:A(\cS^{\gg[t]}) \rightarrow A(\cS) = \cD$ whose image is clearly $\cD^G$ since $h(\tilde{f}_i) =f_i$ up to lower order corrections. We claim that $h$ has trivial kernel, and is therefore an isomorphism from $A(\cS^{\gg[t]})$ to $\cD^G$. To see this, let $r\in \text{Ker}(h)$ be an element of degree $d$, regarded as a polynomial among the $\tilde{f}_i$'s. Under $h$, it maps to a relation in $\cD^G$ whose leading term is the same, with $\tilde{f}_i$ replaced by $f_i$. There is an analogous relation $r(z)\in \cS^{\gg[t]}$, which is obtained up to lower order correction by replacing each $f_i$ with $f_i(z)$ and replacing ordinary products with Wick products. Since $r(z)$ is identically zero, and $\pi_{\text{Zhu}}(r(z)) \in A(\cS^{\gg[t]}$ has the same leading term as $r$, the claim follows by induction on degree. \end{proof}

It follows that in the case where $G = SL_n \times SL_n$ and $V$ is the space of $n\times n$ matrices, $A(\cS^{\gg[t]}) = \cD^G$. Thus, the irreducible, admissible modules over $\cC = \text{Com}(\cA,\cW)$ are precisely the irreducible $\cD^G$-modules. The structure and representation theory of rings of invariant differential operators is an important classical problem. In general, it is believed that these algebras have many features in common with universal enveloping algebras. They have been well studied in the case where $G$ is abelian \cite{MV}, but much less is known for nonabelian groups. The first step in this direction was taken by Schwarz \cite{Sch} in the case where $G = SL_3$ and $V$ is the adjoint module.

In our case, where $V$ is the space of $n\times n$ matrices and $G = SL_n\times SL_n$, $\cD^G$ is closely related to $U(\gs\gl_2)$. It will be convenient to work with the generating set $d, d', c_1,\dots, c_n$ for $\cD^G$, which satisfy relations \begin{equation} \label{reldg} dd' + P(c_1,\dots, c_n)=0,\ \ \ \ \ d'd + Q(c_1,\dots, c_n)=0.\end{equation} Here $P$ and $Q$ are inhomogeneous polynomials whose leading terms coincide with the polynomial $p$ given by (\ref{gradedrel}). Modulo an explicit formula for $P$ and $Q$, we will classify the irreducible, finite dimensional modules over $\cD^G$. These modules are in one-to-one correspondence with the irreducible, admissible $\cC$-modules $M = \bigoplus_{k\geq 0} M_k$ for which each $M_k$ is finite dimensional. In the cases $n=2,3,4$, we will write down $P$ and $Q$, giving an explicit classification of such modules in these cases.

First, we define a notion of Verma module for $\cD^G$. Recall that the center $Z(\cD^G)\subset \cD^G$ is generated by $c_2,\dots, c_n$. Let $A\subset \cD^G$ be the subalgebra generated by $d', c_1,\dots, c_n$. Fix $a\in \mathbb{C}$ and a central character $\lambda$ given by $\lambda(c_i) = \lambda_i$ for $i=2,\dots, n$, and let $C_{a,\lambda}$ be the one-dimensional $A$-module with basis $v_{a,\lambda}$ on which $d'$ acts by zero, $c_1$ acts by $a\cdot \text{id}$, and $c_i$ acts by $\lambda_i\cdot \text{id}$, for $i=2,\dots, n$. Let $$V_{a,\lambda} = \cD^G \otimes_A C_{a,\lambda},$$ which is the $\cD^G$-module spanned by elements of the form $d^k v_{a,\lambda}$, $k\geq 0$. Note that $c_i$ acts by $\lambda_i$ on $V_{a,\lambda}$ for $i=2,\dots, n$, and $V_{a,\lambda} = \bigoplus_{k\in \mathbb{Z}} V_{a,\lambda}[a-kn]$, where $V_{a,\lambda}[a-kn]$ is spanned by $d^k v_{a,\lambda}$ and $c_1$ has eigenvalue $a-kn$ on $V_{a,\lambda}[a-kn]$. 

In order for $V_{a,\lambda}$ to be irreducible of dimension $m<\infty$, we need $d^m v_{\lambda} = 0$ and $d^{m-1} v_{\lambda} \neq 0$. Moreover, we need $(d')^{m-1} d^{m-1} v_{a,\lambda} = c v_{a,\lambda}$ for some $c\neq 0$. We have 
\begin{equation}
\begin{split}
 d d' v_{a,\lambda} &=-P(c_1,\dots, c_n) v_{a,\lambda} = 0\,,\\
 d'd^m v_{a,\lambda} &= (d'd) d^{m-1} v_{a,\lambda}= -Q(c_1,\dots,c_n) d^{m-1} v_{a,\lambda} =0\,,
\end{split} 
 \end{equation} 
 which gives the following algebraic conditions on $a$ and $\lambda$: 
 \begin{equation} \label{firstcond}
 P(a,\lambda_2, \dots, \lambda_n)= 0,\ \ \ \ \ Q(a - n(m-1), \lambda_2,\dots, \lambda_n) = 0.
 \end{equation} 
 The condition $(d')^{m-1} d^{m-1} v_{a,\lambda} = c v_{a,\lambda}$ implies that 
 \begin{equation}\label{secondcond} 
\prod_{i=2}^m Q(a-n(m-i),\lambda_2,\dots,\lambda_n)\neq 0,
 \end{equation} 
 which implies that $Q(a-n(m-i),\lambda_2,\dots,\lambda_n) \neq 0$ for all $i=2,\dots,m$.

In the next three subsections, we write down the polynomials $P$ and $Q$ explicitly in the cases $n=2,3,4$. First, we fix once and for all generators $D, D', C_1,\dots, C_n$ for $\cS^{\gg[t]}$ in these cases. In terms of the usual root bases for $\gs\gl_n$, the generators $C_i$ for $i=2,\dots, n$ are written down in the Appendix. With these choices, we take $d_i = \pi_{\text{Zhu}}(D_i)$, $d'_i = \pi_{\text{Zhu}}(D'_i)$, and $c_i = \pi_{\text{Zhu}}(C_i)$ to be our generators for $\cD^G$. For $i=1,2$, $c_i$ agrees with our previous choice $c_i = \tau(\zeta_i)$ up to scalar multiples, and for $i\geq 3$, $c_i$ agrees with a multiple of $\tau(\zeta_i)$ up to lower order corrections in the Bernstein filtration.

\subsection{The case $n=2$}
The generators $D, D', C_1, C_2$ of $\cS^{\gg[t]}$ satisfy the following normally ordered polynomial relations:
$$:DD': + \frac{1}{2} C_2 - \frac{1}{4} :C_1C_1: -\frac{1}{2} \partial C_1 = 0,$$
$$:D'D: + \frac{1}{2} C_2 - \frac{1}{4} :C_1C_1: +\frac{1}{2} \partial C_1 = 0.$$
Applying the Zhu map yields the following relations among the generators $d, d', c_1, c_2\in \cD^G$: 
$$dd'+\frac{1}{2} c_2 - \frac{1}{4} c_1^2-\frac{3}{2} c_1 -2 = 0,$$
$$d'd+\frac{1}{2} c_2 - \frac{1}{4} c_1^2-\frac{1}{2} c_1  = 0,$$ so we have $P = \frac{1}{2} c_2 - \frac{1}{4} c_1^2-\frac{3}{2} c_1 -2$ and $Q = \frac{1}{2} c_2 - \frac{1}{4} c_1^2-\frac{1}{2} c_1$. We see that $\cD^G\cong U(\gs\gl_2)$ with $[c_1,d'] = 2d'$, $[c_1,d] = -2d$, and $[d,d']= c_1 +2$. The conditions (\ref{firstcond}) in this case are 
$$\frac{1}{2} \lambda_2 - \frac{1}{4} a^2-\frac{3}{2} a -2 = 0,$$
$$\frac{1}{2} \lambda_2 - \frac{1}{4} (a-2m+2)^2-\frac{1}{2} (a-2m+2)  = 0,$$
from which we obtain $a = -3+m$ and $\lambda_2 = \frac{1}{2}(m^2-1)$.
The condition (\ref{secondcond}) is empty in this case, and we obtain the usual classification of finite dimensional $\gs\gl_2$-modules.

\subsection{The case $n=3$}
The relations among the generators $D, D', C_1, C_2,C_3$ of $\cS^{\gg[t]}$ are
$$:D D': - \frac{1}{27} C_3 + \frac{1}{6} :C_2 C_1: - \frac{1}{27} :C_1 C_1 C_1: - \frac{1}{3} : \partial C_1 C_1: - \frac{1}{3} \partial^2 C_1 =0,$$
$$:D' D: - \frac{1}{27} C_3 + \frac{1}{6} :C_2 C_1: - \frac{1}{27} :C_1 C_1 C_1:+ \frac{1}{3} :\partial C_1 C_1: - \frac{1}{2} \partial C_2 - \frac{1}{3} \partial^2 C_1 = 0.$$

The corresponding relations in $\cD^G$ are
$$dd' - \frac{1}{27} c_3 + \frac{1}{6} c_2 c_1 -\frac{1}{27} c_1^3 + \frac{3}{2} c_2 -\frac{2}{3} c_1^2-\frac{11}{3} c_1 -6=0,$$
$$d'd - \frac{1}{27} c_3 + \frac{1}{6} c_2 c_1-\frac{1}{27} c_1^3 + c_2 -\frac{1}{3} c_1^2-\frac{2}{3} c_1=0,$$ so 
\begin{equation*}
\begin{split}
P\, &=\, - \frac{1}{27} c_3 + \frac{1}{6} c_2 c_1 -\frac{1}{27} c_1^3 + \frac{3}{2} c_2 -\frac{2}{3} c_1^2-\frac{11}{3} c_1 -6\,,\\
Q\, &=\,  -\frac{1}{27} c_3 + \frac{1}{6} c_2 c_1-\frac{1}{27} c_1^3 + c_2 -\frac{1}{3} c_1^2-\frac{2}{3} c_1\,.
\end{split}
\end{equation*}
The conditions (\ref{firstcond}) in this case are
$$- \frac{1}{27} \lambda_3 + \frac{1}{6} \lambda_2 a-\frac{1}{27} a^3 +\frac{3}{2} \lambda_2 -\frac{2}{3} a^2-\frac{11}{3} a-6=0,$$
$$- \frac{1}{27} \lambda_3 + \frac{1}{6} \lambda_2 (a-3m+3)-\frac{1}{27} (a-3m+3)^3 - \frac{1}{3}  (a - 3 m + 3)^2 + \lambda_2 - \frac{2}{3} (a - 3 m + 3)=0.$$ This yields $$\lambda_2= \frac{2}{3}(33 + 12 a + a^2 - 18 m - 3 a m + 3 m^2),  \ \ \ \ \lambda_3 = (9 + a) (81 + 27 a + 2 a^2 - 54 m - 9 a m + 9 m^2),$$ where $a$ can be arbitrary. Finally, condition (\ref{secondcond}) implies that for each fixed $m$, a finite set of values of $a, \lambda$ must be excluded in order for $V_{a,\lambda}$ to be irreducible. For example, for $i=2\leq m$ we solve $$- \frac{1}{27} \lambda_3 + \frac{1}{6} \lambda_2 (a-3m+3i)-\frac{1}{27} (a-3m+3i)^3 - \frac{1}{3}  (a - 3 m + 3i)^2 + \lambda_2 - \frac{2}{3} (a - 3 m + 3i)\neq0$$ to obtain $$(a,\lambda_2, \lambda_3) \neq (-7+2m, \frac{2}{3}(-2-m+m^2), -2(10+3m-6m^2+m^3)),$$ and for $i=3\leq m$ we obtain $$(a,\lambda_2, \lambda_3) \neq (-8+2m, \frac{2}{3}(1-2m+m^2), -7-6m+15m^2-2m^3).$$

\subsection{The case $n=4$}
The relations among the generators $D, D', C_1, C_2,C_3,C_4$ of $\cS^{\gg[t]}$ are
$$:D D': - \frac{1}{256} C_4 + \frac{1}{32} :C_1 C_3: + \frac{1}{32} :C_2 C_1  C_1: - 
 \frac{1}{256} :C_1 C_1 C_1 C_1: + \frac{1}{8} :\partial C_1 C_2: $$ $$ -\frac{3}{32} : \partial C_1 C_1 C_1: - 
 \frac{1}{4} :\partial^2 C_1 C_1: - \frac{3}{16} :\partial C_1 \partial C_1: - \frac{1}{4} \partial^3 C_1 =0,$$
$$:D'D: - \frac{1}{256} C_4 + \frac{1}{32} :C_1 C_3: + \frac{1}{32} :C_2  C_1  C_1: - 
 \frac{1}{256} :C_1  C_1 C_1 C_1: - \frac{1}{8} \partial C_3 - \frac{1}{8} :\partial C_1 C_2: $$ $$- \frac{1}{4} :C_1\partial C_2: +  \frac{3}{32} :\partial C_1 C_1 C_1: + \frac{1}{2} \partial^2 C_2 - \frac{1}{4}: \partial^2 C_1 C_1: - \frac{3}{16} :\partial C_1 \partial C_1: +  \frac{1}{4} \partial^3 C_1 =0.$$ 
 
Applying the Zhu map yields
$$dd' -\frac{1}{256} c_4 + \frac{1}{32}  c_1 c_3 + \frac{1}{32}  c_2 c_1^2 - \frac{1}{256}  c_1^4 +  \frac{1}{2} c_3 + \frac{7}{8} c_2 c_1 - \frac{5}{32} c_1^3 + 6 c_2 - \frac{35}{16} c_1^2 - \frac{25}{2} c_1 - 24=0,$$  
$$d'd -\frac{1}{256} c_4 + \frac{1}{32}  c_1 c_3 + \frac{1}{32}  c_2 c_1^2 - \frac{1}{256}  c_1^4 + \frac{3}{8} c_3 + \frac{5}{8} c_1 c_2 - \frac{3}{32}c_1^3 + 3 c_2 - \frac{11}{16} c_1^2 -  \frac{3}{2} c_1=0,$$ so we have $$P = -\frac{1}{256} c_4 + \frac{1}{32}  c_1 c_3 + \frac{1}{32}  c_2 c_1^2 - \frac{1}{256}  c_1^4 +  \frac{1}{2} c_3 + \frac{7}{8} c_2 c_1 - \frac{5}{32} c_1^3 + 6 c_2 - \frac{35}{16} c_1^2 - \frac{25}{2} c_1 - 24,$$ $$Q=-\frac{1}{256} c_4 + \frac{1}{32}  c_1 c_3 + \frac{1}{32}  c_2 c_1^2 - \frac{1}{256}  c_1^4 + \frac{3}{8} c_3 + \frac{5}{8} c_1 c_2 - \frac{3}{32}c_1^3 + 3 c_2 - \frac{11}{16} c_1^2 -  \frac{3}{2} c_1.$$

Conditions (\ref{firstcond}) in this case are $$-\frac{1}{256} \lambda_4 + \frac{1}{32}  a \lambda_3 + \frac{1}{32}  \lambda_2 a^2 - \frac{1}{256}  a^4 + \frac{1}{2} \lambda_3 + \frac{7}{8} \lambda_2 a - \frac{5}{32} a^3 + 6 \lambda_2 - \frac{35}{16} a^2 - \frac{25}{2} a - 24=0,$$  
$$-\frac{1}{256} \lambda_4 + \frac{1}{32}   \lambda_3(a - 4 m + 4) + \frac{1}{32}  \lambda_2 (a - 4 m + 4)^2 - \frac{1}{256}  (a - 4 m + 4)^4 + \frac{3}{8} \lambda_3 + \frac{5}{8}  \lambda_2 (a - 4 m + 4)$$ $$ - \frac{3}{32}(a - 4 m + 4)^3 + 3 \lambda_2 - \frac{11}{16} (a - 4 m + 4)^2 -  \frac{3}{2} (a - 4 m + 4) = 0.$$     

We obtain $$ \lambda_3 = \frac{1}{2}(800 + 280 a + 30 a^2 + a^3 - 56 \lambda_2 - 4 a \lambda_2 - 560 m - 120 a m - 6 a^2 m + 8 \lambda_2 m + 160 m^2 + 16 a m^2 - 16 m^3), $$ \begin{equation} \label{condi4} \lambda_4 = (16 + a) (16 + a - 4 m) (176 + 48 a + 3 a^2 - 8 \lambda_2 - 96 m -12 a m + 16 m^2),\end{equation} where $a$ and $\lambda_2$ can be arbitrary. Finally, condition (\ref{secondcond}) is $$-\frac{1}{256} \lambda_4 + \frac{1}{32}   \lambda_3(a - 4 m + 4i) + \frac{1}{32}  \lambda_2 (a - 4 m + 4i)^2 - \frac{1}{256}  (a - 4 m + 4i)^4 + \frac{3}{8} \lambda_3 + \frac{5}{8}  \lambda_2 (a - 4 m + 4i) $$ \begin{equation} \label{condii4}- \frac{3}{32}(a - 4 m + 4i)^3 + 3 \lambda_2 - \frac{11}{16} (a - 4 m + 4i)^2 -  \frac{3}{2} (a - 4 m + 4i) \neq 0\end{equation} for $i=2,\dots,m$, which eliminates an algebraic subvariety of dimension one. It follows that the set of $(a,\lambda)$ which satisfy (\ref{condi4}) will also satisfy (\ref{condii4}) generically.

\section{Appendix}

In this Appendix, we write down explicit formulas for $C_2,\dots, C_n$ in the case $n=3$ and $n=4$. We work in the standard root bases of $\gs\gl_3$ and $\gs\gl_4$, where $L^{ij}$ corresponds to the matrix $E_{ij}$ for $i\neq j$, and $L^{H_i}$ corresponds to $E_{11} - E_{i+1,i+1}$. These calculations were performed using the Mathematica package of K. Thielemans \cite{T}.

In the case $n=3$, we have
$$C_2 = : L^{12}  L^{21}: + : L^{21}   L^{12}: + : L^{13}  L^{31}: + : L^{31}  L^{13}: +: L^{23}   L^{32}: + : L^{32}  L^{23}: $$ $$+ 2/3 : L^{H_{1}}  L^{H_{1}}:  - 2/3 : L^{H_{1}}  L^{H_{2}}: + 2/3 : L^{H_{2}}  L^{H_{2}}:$$

$$C_3 = -27 : L^{12}  L^{23}  L^{31}: - 27 : L^{13}  L^{21}  L^{32}: - 
  18 : L^{13}  L^{31}  L^{H_{1}}: + 9 : L^{23}  L^{32}   L^{H_{1}}: $$ $$- 
  18 : L^{12}  L^{21}   L^{H_{2}}: + 9 : L^{23}  L^{32}   L^{H_{2}}: + 
  9 : L^{12}  L^{21}   L^{H_{1}}: + 9 : L^{13}  L^{31}   L^{H_{2}}: $$ $$- 
  3 : L^{H_{1}}  L^{H_{1}}   L^{H_{2}}: - 3 : L^{H_{1}}  L^{H_{2}}   L^{H_{2}}: + 
  2 : L^{H_{1}}  L^{H_{1}}   L^{H_{1}}: + 2 : L^{H_{2}}  L^{H_{2}}   L^{H_{2}}: $$ $$- 
  27 : L^{13}   \partial L^{31}: - 27 : L^{23}   \partial L^{32}: - 27 : L^{32}  \partial L^{23}: - 
  9 : L^{H_{1}} \partial  L^{H_{1}}: + 9 : L^{H_{1}} \partial  L^{H_{2}}: + 18 : L^{H_{2}} \partial  L^{H_{1}}: $$ $$- 
  18 : L^{H_{2}} \partial  L^{H_{2}}: + 9/2  \partial^2 L^{H_{1}} + 9/2 \partial^2 L^{H_{2}}.$$

In the case $n=4$, we have

$$ C_2 = : L^{12}   L^{21}: + : L^{21}   L^{12}: + : L^{13}   L^{31}: + : L^{31}   L^{13}: + : L^{14}   L^{41}: + : L^{41}   L^{14}: $$ $$+ : L^{23}   L^{32}: + : L^{32}  L^{23}: +  : L^{24}   L^{42}: + : L^{42} L^{24}: + : L^{34}   L^{43}: + : L^{43}   L^{34}: $$ $$ + 3/4 : L^{H_{1}}   L^{H_{1}}: + 3/4 : L^{H_{2}}   L^{H_{2}}: + 3/4 : L^{H_{3}}   L^{H_{3}}: - 1/2 : L^{H_{1}}  L^{H_{2}}: - 1/2 : L^{H_{1}}   L^{H_{3}}: - 1/2 : L^{H_{2}}   L^{H_{3}}:$$

$$C_3 = -4 : L^{12}  L^{21}   L^{H_{1}}: + 4 : L^{12}  L^{21}   L^{H_{2}}: + 
  4 : L^{12}  L^{21}   L^{H_{3}}: + 8 : L^{12}  L^{23 }  L^{31}: $$ $$+ 
  8 : L^{12}  L^{24}   L^{41}: + 8 : L^{13}  L^{21}   L^{32}: + 
  4 : L^{13}  L^{31}   L^{H_{1}}: - 4 : L^{13}  L^{31}   L^{H_{2}}: $$ $$+ 
  4 : L^{13}  L^{31}   L^{H_{3}}: + 8 : L^{13}  L^{34}   L^{41}: + 
  8 : L^{14}  L^{21}   L^{42}: + 8 : L^{14}  L^{31}   L^{43}: $$ $$+ 
  4 : L^{14}  L^{41}  L^{H_{1}}: + 4 : L^{14}  L^{41}   L^{H_{2}}: - 
  4 : L^{14}  L^{41}   L^{H_{3}}: - 4 : L^{23}  L^{32}   L^{H_{1}}:$$ $$ - 
  4 : L^{23}  L^{32}   L^{H_{2}}: + 4 : L^{23}  L^{32}   L^{H_{3}}: + 
  8 : L^{23}  L^{34}   L^{42}: + 8 : L^{24}  L^{32}   L^{43}: $$ $$- 
  4 : L^{24}  L^{42}   L^{H_{1}}: + 4 : L^{24}  L^{42}   L^{H_{2}}: - 
  4 : L^{24}  L^{42}   L^{H_{3}}: + 4 : L^{34}  L^{43}   L^{H_{1}}: $$ $$- 
  4 : L^{34}  L^{43}   L^{H_{2}}: - 4 : L^{34}  L^{43}   L^{H_{3}}: - 
  : L^{H_{1}}  L^{H_{1}}   L^{H_{1}}: + : L^{H_{1}}  L^{H_{1}}   L^{H_{2}}: $$ $$+ 
  : L^{H_{1}}  L^{H_{1}}   L^{H_{3}}: + : L^{H_{1}}  L^{H_{2}}   L^{H_{2}}: - 
  2 : L^{H_{1}}  L^{H_{2}}   L^{H_{3}}: + : L^{H_{1}}  L^{H_{3}}   L^{H_{3}}: $$ $$ - 
  : L^{H_{2}}  L^{H_{2}}   L^{H_{2}}: + : L^{H_{2}}  L^{H_{2}}   L^{H_{3}}: + 
  : L^{H_{2}}  L^{H_{3}}   L^{H_{3}}: - : L^{H_{3}}  L^{H_{3}}   L^{H_{3}}: $$ $$+ 8 : L^{13}   \partial L^{31}: + 
  16 : L^{14}  \partial  L^{41}: + 8 : L^{23}   \partial L^{32}: + 8 : L^{32} \partial  L^{23}: + 
  16 : L^{24}   \partial L^{42}: + 8 : L^{42} \partial   L^{24} : $$ $$+ 16 : L^{34}  \partial  L^{43}: + 
  16 : L^{43}  \partial L^{34}: + 4 : L^{H_{1}}\partial    L^{H_{1}}: - 4 : L^{H_{1}}\partial  L^{H_{3}}: - 
  4 : L^{H_{2}}\partial L^{H_{1}}: + 8 : L^{H_{2}}  \partial  L^{H_{2}}: $$ $$ - 4 : L^{H_{2}} \partial L^{H_{3}}: - 
  4 : L^{H_{3}} \partial   L^{H_{1}}: - 8 : L^{H_{3}} \partial    L^{H_{2}}: + 12 : L^{H_{3}}\partial   L^{H_{3}}: - 4  \partial^2 L^{H_{2}} -  4 \partial^2 L^{H_{3}}, $$
 
$$C_4 = -16 : L^{12}  L^{21}   L^{H_{1}}    L^{H_{1}}: + 32  : L^{12}    L^{21}   L^{H_{1}}    L^{H_{2}}: + 
  32  : L^{12}   L^{21}    L^{H_{1}}   L^{H_{3}}: + 48 : L^{12}   L^{21}   L^{H_{2}}  L^{H_{2}}: $$ $$- 
  160  : L^{12}   L^{21}   L^{H_{2}}    L^{H_{3}}: + 48 : L^{12}   L^{21}  L^{H_{3}}   L^{H_{3}}: + 
  64  : L^{12}    L^{23}   L^{31}    L^{H_{1}}: + 64 : L^{12}    L^{23}   L^{31}   L^{H_{2}}: $$ $$- 
  192  : L^{12}    L^{23}   L^{31}   L^{H_{3}}: - 256 : L^{12}   L^{23}   L^{34}   L^{41}: - 
  256 : L^{12}   L^{24}   L^{31}    L^{43}: + 64 : L^{12}    L^{24}   L^{41}   L^{H_{1}}: $$ $$- 
  192 : L^{12}   L^{24}  L^{41}   L^{H_{2}}: + 64 : L^{12}    L^{24}    L^{41}    L^{H_{3}}: + 
  256  : L^{12}    L^{34}    L^{21}   L^{43}: + 64 :  L^{13}   L^{21}   L^{32}   L^{H_{1}}: $$ $$+ 
  64  : L^{13}   L^{21}  L^{32}   L^{H_{2}}: - 192 :  L^{13}    L^{21}  L^{32}    L^{H_{3}}: + 
  256 : L^{13}   L^{24}    L^{31}   L^{42}: - 256 :  L^{13}    L^{24}   L^{41}   L^{32}: $$ $$+ 
  48  : L^{13}   L^{31}   L^{H_{1}}   L^{H_{1}}: + 32  : L^{13}    L^{31}   L^{H_{1}}    L^{H_{2}}: - 
  160  : L^{13}   L^{31}   L^{H_{1}}    L^{H_{3}}: - 16  : L^{13}    L^{31}   L^{H_{2}}   L^{H_{2}}: $$ $$+ 
  32  : L^{13}   L^{31}    L^{H_{2}}    L^{H_{3}}: + 48 :  L^{13}    L^{31}    L^{H_{3}}   L^{H_{3}}: - 
  256 : L^{13}    L^{34}    L^{21}    L^{42}: - 192 : L^{13}     L^{34}    L^{41}    L^{H_{1}}: $$ $$+ 
  64  : L^{13}   L^{34}   L^{41}  L^{H_{2}}: + 64 : L^{13}    L^{34}   L^{41}    L^{H_{3}}: - 
  256  : L^{14}   L^{21}    L^{32}   L^{43}: + 64 : L^{14}   L^{21}  L^{42}   L^{H_{1}}: $$ $$- 
  192  : L^{14}    L^{21}    L^{42}    L^{H_{2}}: + 64  : L^{14}   L^{21}   L^{42}   L^{H_{3}}: - 
  256  : L^{14}    L^{23}   L^{31}   L^{42}: + 256  : L^{14}    L^{23}   L^{41}   L^{32}: $$ $$- 
  192  : L^{14}   L^{31}   L^{43}    L^{H_{1}}: + 64  : L^{14}    L^{31}   L^{43}    L^{H_{2}}: + 
  64  : L^{14}   L^{31}    L^{43}   L^{H_{3}}: + 48  : L^{14}   L^{41}   L^{H_{1}}   L^{H_{1}}: $$ $$- 
  160  : L^{14}    L^{41}   L^{H_{1}}   L^{H_{2}}: + 32  : L^{14}   L^{41}   L^{H_{1}}   L^{H_{3}} : + 
  48 : L^{14}    L^{41}    L^{H_{2}}    L^{H_{2}}: + 32 : L^{14}   L^{41}   L^{H_{2}}    L^{H_{3}}: $$ $$- 
  16  : L^{14}   L^{41}   L^{H_{3}}    L^{H_{3}}: - 16 :  L^{23}   L^{32}   L^{H_{1}}   L^{H_{1}}: - 
  32  : L^{23}   L^{32}    L^{H_{1}}    L^{H_{2}}: + 32  : L^{23}   L^{32}    L^{H_{1}}   L^{H_{3}}: $$ $$- 
  16  : L^{23}   L^{32}   L^{H_{2}}    L^{H_{2}}: + 32  : L^{23}    L^{32}   L^{H_{2}}    L^{H_{3}} : + 
  48 : L^{23}   L^{32}   L^{H_{3}}   L^{H_{3}}: + 64  : L^{23}   L^{34}   L^{42}   L^{H_{1}} : $$ $$+ 
  64 : L^{23}    L^{34}    L^{42}    L^{H_{2}} : + 64  : L^{23}   L^{34}    L^{42}    L^{H_{3}} : + 
  64  : L^{24}   L^{32}    L^{43}   L^{H_{1}} : + 64 :  L^{24}   L^{32}    L^{43}   L^{H_{2}}: $$ $$+ 
  64 : L^{24}   L^{32}   L^{43}    L^{H_{3}} : - 16 :  L^{24}   L^{42}   L^{H_{1}}   L^{H_{1}}: + 
  32 :  L^{24}   L^{42}   L^{H_{1}}   L^{H_{2}}: - 32  : L^{24}   L^{42}    L^{H_{1}}  L^{H_{3}} : $$ $$+ 
  48 : L^{24}   L^{42}   L^{H_{2}}   L^{H_{2}} : + 32  : L^{24}    L^{42}    L^{H_{2}}   L^{H_{3}} : - 
  16 :  L^{24}    L^{42}   L^{H_{3}}  L^{H_{3}}: + 48 :  L^{34}    L^{43}   L^{H_{1}}   L^{H_{1}} : $$ $$+ 
  32 : L^{34}    L^{43}   L^{H_{1}}   L^{H_{2}} : + 32  : L^{34}    L^{43}    L^{H_{1}}    L^{H_{3}} : - 
  16  : L^{34}    L^{43}  L^{H_{2}}   L^{H_{2}}: - 32  : L^{34}    L^{43}    L^{H_{2}}    L^{H_{3}}: $$ $$- 
  16 : L^{34}    L^{43}   L^{H_{3}}   L^{H_{3}}: - 3 :  L^{H_{1}}    L^{H_{1}}   L^{H_{1}}   L^{H_{1}} : + 
  4 :  L^{H_{1}}   L^{H_{1}}   L^{H_{1}}    L^{H_{2}} : + 4  : L^{H_{1}}   L^{H_{1}}   L^{H_{1}}   L^{H_{3}}: $$ $$+ 
  14 : L^{H_{1}}    L^{H_{1}}    L^{H_{2}}    L^{H_{2}}: - 20  : L^{H_{1}}   L^{H_{1}}   L^{H_{2}}   L^{H_{3}} : + 
  14  : L^{H_{1}}   L^{H_{1}}   L^{H_{3}}   L^{H_{3}} : + 4  : L^{H_{1}}    L^{H_{2}}   L^{H_{2}}   L^{H_{2}} :$$ $$ - 
  20 : L^{H_{1}}    L^{H_{2}}   L^{H_{2}}   L^{H_{3}} : - 20 : L^{H_{1}}   L^{H_{2}}    L^{H_{3}}   L^{H_{3}} : + 
  4 : L^{H_{1}}   L^{H_{3}}   L^{H_{3}}    L^{H_{3}} : - 3  : L^{H_{2}}    L^{H_{2}}   L^{H_{2}}   L^{H_{2}} : $$ $$+ 
  4 : L^{H_{2}}   L^{H_{2}}  L^{H_{2}}    L^{H_{3}} : + 14  : L^{H_{2}}    L^{H_{2}}   L^{H_{3}}    L^{H_{3}} : + 
  4 : L^{H_{2}}   L^{H_{3}}    L^{H_{3}}    L^{H_{3}}: - 3  : L^{H_{3}}   L^{H_{3}}    L^{H_{3}}    L^{H_{3}}: $$ $$+
64 : L^{12}  L^{21}  \partial L^{H_{1}}: + 64 : L^{12}  L^{21} \partial  L^{H_{2}}: - 
 192 : L^{12}  L^{21} \partial   L^{H_{3}}: + 64 : L^{13} \partial  L^{31}  L^{H_{1}}: $$ $$+ 
 64 : L^{13}  L^{31} \partial  L^{H_{1}}: + 64 : L^{13} \partial L^{31}  L^{H_{2}}: + 
 64 : L^{13}  L^{31} \partial  L^{H_{2}}: - 192 : L^{13} \partial L^{31}   L^{H_{3}}: $$ $$- 
 192 : L^{13}  L^{31} \partial  L^{H_{3}}: - 128 : L^{14} \partial  L^{41}   L^{H_{1}}: + 
 64 : L^{14}  L^{41} \partial  L^{H_{1}}: - 128 : L^{14} \partial L^{41}   L^{H_{2}}: $$ $$- 
 192 : L^{14}  L^{41}  \partial L^{H_{2}}: + 128 : L^{14} \partial  L^{41}  L^{H_{3}}: + 
 64 : L^{14} L^{41} \partial   L^{H_{3}}: + 64 : \partial L^{23}  L^{32}   L^{H_{1}}: $$ $$+ 
 64 : L^{23} \partial  L^{32}   L^{H_{1}}: + 64 : L^{23}  L^{32} \partial  L^{H_{1}}: + 
 64 : \partial L^{23}  L^{32}   L^{H_{2}}: + 64 : L^{23} \partial  L^{32}   L^{H_{2}}: $$ $$+ 
 64 : L^{23}  L^{32} \partial   L^{H_{2}}: - 192 : \partial L^{23} L^{32}   L^{H_{3}}: - 
 192 : L^{23}  \partial L^{32}   L^{H_{3}}: - 192 : L^{23}  L^{32}  \partial L^{H_{3}}: $$ $$+ 
 64 : \partial L^{24}  L^{42}   L^{H_{1}}: + 128 : L^{24} \partial  L^{42}  L^{H_{1}}: + 
 64 : L^{24}  L^{42} \partial   L^{H_{1}}: - 192 : \partial L^{24}  L^{42}   L^{H_{2}}: $$ $$- 
 128 : L^{24} \partial L^{42}   L^{H_{2}}: - 192 : L^{24}  L^{42}   \partial L^{H_{2}}: + 
 64 : \partial L^{24}  L^{42}   L^{H_{3}}: + 128 : L^{24} \partial L^{42}   L^{H_{3}}: $$ $$+ 
 64 : L^{24}  L^{42}  \partial L^{H_{3}}: - 128 : \partial L^{34}  L^{43}   L^{H_{1}}: - 
 128 : L^{34} \partial  L^{43}   L^{H_{1}}: - 192 : L^{34}  L^{43}   \partial L^{H_{1}}: $$ $$+ 
 128 : \partial L^{34}  L^{43}   L^{H_{2}}: + 128 : L^{34} \partial  L^{43}   L^{H_{2}}: + 
 64 : L^{34}  L^{43}  \partial L^{H_{2}}: + 128 : \partial L^{34}  L^{43}   L^{H_{3}}: $$ $$+ 
 128 : L^{34} \partial L^{43}   L^{H_{3}}: + 64 : L^{34}  L^{43}  \partial L^{H_{3}}: - 
 256 : L^{12}  L^{24}  \partial L^{41}: - 256 : L^{13} \partial  L^{34}   L^{41}: $$ $$- 
 256 : L^{13}  L^{34}  \partial L^{41}: - 256 : L^{14}  L^{21} \partial  L^{42}: - 
 256 : L^{14} \partial  L^{31}  L^{43}: - 256 : L^{14}  L^{31}  \partial L^{43}:$$ $$ - 
 256 : \partial L^{23} L^{34}   L^{42}: - 256 : L^{23}  \partial L^{34}   L^{42}: - 
 256 : L^{23}  L^{34}  \partial L^{42}: - 256 : \partial L^{24} L^{32}   L^{43}: $$ $$- 
 256 : L^{24} \partial L^{32}   L^{43}: - 256 : L^{24}  L^{32} \partial   L^{43}: + 
 40 : \partial L^{H_{1}}  L^{H_{1}}   L^{H_{1}}: - 48 : \partial L^{H_{1}} L^{H_{1}}   L^{H_{2}}: $$ $$- 
 8 : L^{H_{1}}  L^{H_{1}}  \partial L^{H_{2}}: - 48 : \partial L^{H_{1}}  L^{H_{1}}   L^{H_{3}}: - 
 56 : L^{H_{1}}  L^{H_{1}}  \partial L^{H_{3}}: - 24 : \partial L^{H_{1}}  L^{H_{2}}   L^{H_{2}}:$$ $$ - 
 16 : L^{H_{1}} \partial L^{H_{2}}   L^{H_{2}}: - 24 : \partial L^{H_{1}}  L^{H_{3}}   L^{H_{3}}: - 
 48 : L^{H_{1}} \partial L^{H_{3}}   L^{H_{3}}: + 144 : \partial L^{H_{1}}  L^{H_{2}}   L^{H_{3}}: $$ $$+ 
 112 : L^{H_{1}} \partial L^{H_{2}}  L^{H_{3}}: + 80 : L^{H_{1}}  L^{H_{2}}   \partial L^{H_{3}}: + 
 56 : \partial L^{H_{2}}  L^{H_{2}}   L^{H_{2}}: - 80 : \partial L^{H_{2}}  L^{H_{2}}   L^{H_{3}}:$$ $$ - 
 56 : L^{H_{2}}  L^{H_{2}}   \partial L^{H_{3}}: - 72 : \partial L^{H_{2}} L^{H_{3}}   L^{H_{3}}: - 
 48 : L^{H_{2}} \partial  L^{H_{3}}   L^{H_{3}}: + 72 : \partial L^{H_{3}} L^{H_{3}}   L^{H_{3}}: $$ $$- 
 256 : L^{14}  \partial^2 L^{41}: - 256 : \partial L^{24}  \partial L^{42}: - 256 : L^{24}  \partial^2 L^{42}: - 
 256 : \partial^2 L^{34}  L^{43}: $$ $$- 512 : \partial L^{34} \partial   L^{43}: - 256 : L^{34}\partial^2   L^{43}: - 
 16 : \partial L^{H_{1}}   \partial L^{H_{1}}: - 64 : \partial^2 L^{H_{2}}  L^{H_{2}}: $$ $$- 80 : \partial L^{H_{2}}  \partial L^{H_{2}}: - 
 192 : \partial^2 L^{H_{3}}   L^{H_{3}}: - 144 : \partial L^{H_{3}}   \partial L^{H_{3}}: - 96 : \partial L^{H_{1}} \partial   L^{H_{2}}: $$ $$- 
 64 : L^{H_{1}}   \partial^2 L^{H_{2}}: + 96 :\partial L^{H_{1}} \partial   L^{H_{3}}: + 64 : L^{H_{1}}  \partial L^{H_{3}}: + 
 192 :\partial^2 L^{H_{2}}   L^{H_{3}}: $$ $$+ 288 : \partial L^{H_{2}}  \partial L^{H_{3}}: + 64 : L^{H_{2}} \partial^2   L^{H_{3}}: - 
 64  \partial^3 L^{H_{1}}- 64 \partial^3 L^{H_{2}} + 192  \partial^3 L^{H_{3}}. $$

\end{document}